\documentclass[a4paper,oneside,11pt]{article}

\usepackage[a4paper]{geometry}
\usepackage{aeguill}                 
\usepackage{graphicx}
\usepackage{amsmath,amsfonts,amssymb,amsthm}
\usepackage{pstricks} 
\usepackage{epstopdf}
\usepackage{srcltx}
\usepackage[utf8]{inputenc} 
\usepackage[T1]{fontenc} 
\usepackage{hyperref} 
\usepackage[english]{babel} 
\usepackage{comment}

\title{Proper proximality in non-positive curvature}
\author{Camille Horbez, Jingyin Huang and Jean Lécureux}
\date{\today}
 
\usepackage[all]{xy}
\usepackage{bbm}
\begin{document}
\maketitle 
\newtheorem{de}{Definition} [section]
\newtheorem{theo}[de]{Theorem} 
\newtheorem{prop}[de]{Proposition}
\newtheorem{lemma}[de]{Lemma}
\newtheorem{cor}[de]{Corollary}
\newtheorem{propd}[de]{Proposition-Definition}
\newtheorem{conj}[de]{Conjecture}
\newtheorem*{mtheo}{Theorem}
\newtheorem{fact}[de]{Fact}
\newtheorem{theointro}{Theorem}
\newtheorem{corintro}[theointro]{Corollary}

\theoremstyle{remark}
\newtheorem{rk}[de]{Remark}
\newtheorem{ex}[de]{Example}
\newtheorem{question}[de]{Question}

\normalsize
\newcommand{\Res}{\mathrm{Res}}
\newcommand{\proj}{\mathrm{proj}}
\newcommand{\NN}{\mathbb N}
\newcommand{\RR}{\mathbb R}
\newcommand{\Aut}{\mathrm{Aut}}
\newcommand{\Out}{\mathrm{Out}}
\newcommand{\dunion}{\sqcup}
\newcommand{\eps}{\varepsilon}
\renewcommand{\epsilon}{\varepsilon}
\newcommand{\calf}{\mathcal{F}}
\newcommand{\cali}{\mathcal{I}}
\newcommand{\caly}{\mathcal{Y}}
\newcommand{\calx}{\mathcal{X}}
\newcommand{\calz}{\mathcal{Z}}
\newcommand{\calo}{\mathcal{O}}
\newcommand{\calb}{\mathcal{B}}
\newcommand{\calq}{\mathcal{Q}}
\newcommand{\calu}{\mathcal{U}}
\newcommand{\call}{\mathcal{L}}
\newcommand{\bbR}{\mathbb{R}}
\newcommand{\bbZ}{\mathbb{Z}}
\newcommand{\bbD}{\mathbb{D}}
\newcommand{\NT}{\mathrm{NT}}
\newcommand{\cat}{\mathrm{CAT}(0)}
\newcommand{\actson}{\curvearrowright}
\newcommand{\caln}{\mathcal{N}}
\newcommand{\calg}{\mathcal{G}}
\newcommand{\Prob}{\mathrm{Prob}}
\newcommand{\calt}{\mathcal{T}}
\newcommand{\calc}{\mathcal{C}}
\newcommand{\adm}{\mathrm{adm}}
\newcommand{\cala}{\mathcal{A}}
\newcommand{\cals}{\mathcal{S}}
\newcommand{\calh}{\mathcal{H}}
\newcommand{\Stab}{\mathrm{Stab}}
\newcommand{\Isom}{\mathrm{Isom}}
\newcommand{\cod}{\Phi^\ast}
\newcommand{\horo}{\overline{\cod}^h}
\newcommand{\horox}{\overline{X}^h}
\newcommand{\bdd}{\mathrm{bdd}}
\newcommand{\calp}{\mathcal{P}}
\newcommand{\flex}{\mathrm{flex}}
\newcommand{\bv}{\mathcal{B}V}
\newcommand{\Fix}{\mathrm{Fix}}
\newcommand{\bad}{\mathrm{bad}}
\newcommand{\prox}{\mathrm{profinite}}
\newcommand{\vis}{\mathrm{vis}}
\newcommand{\Mod}{\mathrm{Mod}}
\newcommand{\diam}{\mathrm{diam}}
\newcommand{\op}{\mathrm{op}}
\newcommand{\Conv}{\mathrm{Conv}}
\newcommand{\Supp}{\mathrm{Supp}}
\newcommand{\WZ}{\mathrm{FC}}

\makeatletter
\edef\@tempa#1#2{\def#1{\mathaccent\string"\noexpand\accentclass@#2 }}
\@tempa\rond{017}
\makeatother

\begin{abstract}
Proper proximality of a countable group is a notion that was introduced by Boutonnet, Ioana and Peterson as a tool to study rigidity properties of certain von Neumann algebras associated to groups or ergodic group actions. In the present paper, we establish the proper proximality of many groups acting on nonpositively curved spaces. 

First, these include many countable groups $G$ acting properly nonelementarily by isometries on a proper $\cat$ space $X$. More precisely, proper proximality holds in the presence of rank one isometries or when $X$ is a locally thick affine building with a minimal $G$-action. As a consequence of Rank Rigidity, we derive the proper proximality of all countable nonelementary $\cat$ cubical groups, and of all countable groups acting properly cocompactly nonelementarily by isometries on either a Hadamard manifold with no Euclidean factor, or on a $2$-dimensional piecewise Euclidean $\cat$ simplicial complex.

Second, we establish the proper proximality of many hierarchically hyperbolic groups. These include the mapping class groups of connected orientable finite-type boundaryless surfaces (apart from a few low-complexity cases), thus answering a question raised by Boutonnet, Ioana and Peterson. We also prove the proper proximality of all subgroups acting nonelementarily on the curve graph.

In view of work of Boutonnet, Ioana and Peterson, our results have applications to structural and rigidity results for von Neumann algebras associated to all the above groups and their ergodic actions.
\end{abstract}

\section*{Introduction}

Motivated by questions concerning the rigidity of von Neumann algebras associated to groups or group actions, Boutonnet, Ioana and Peterson introduced in \cite{BIP} the notion of \emph{proper proximality} of a countable group, generalizing the notion of \emph{bi-exactness} from \cite[Section~15]{BO}. They gave many examples  \cite[Proposition~1.6]{BIP}, including all nonelementary convergence groups and all lattices in noncompact semi-simple Lie groups. This enabled them to obtain the first strong $W^*$-rigidity results for compact actions of $\mathrm{SL}_d(\mathbb{Z})$ with $d\ge 3$. More recently, Ishan, Peterson and Ruth proved that proper proximality is stable under measure equivalence and $W^*$-equivalence, and as a consequence they obtained new examples of properly proximal groups \cite{IPR}. 

A possible definition is as follows. A countable group $G$ is \emph{properly proximal} if there exist a compact $G$-space $K$ that does not carry any $G$-invariant probability measure, and a diffuse probability measure $\eta$ on $K$, such that for every nonprincipal ultrafilter $\omega$ on $G$ and every $h\in G$, one has $\lim\limits_{g\to\omega}((gh)\cdot\eta-g\cdot\eta)=0$ in the weak-$*$ topology. The definition adopted in \cite{BIP} and in Section~\ref{sec:proper-proximality} below involves considering finitely many compact spaces instead of just one, which is helpful in applications, but these definitions are equivalent by \cite[Theorem~4.3]{BIP}. Geometric criteria for checking proper proximality, phrased in a language that might sound more familiar to geometric group theorists, are provided in Section~\ref{sec:proper-proximality} below, and discussed later in this introduction -- under the heading \emph{A word on the proofs}. 

In the present paper, we establish the proper proximality of various classes of groups that satisfy a form of non-positive curvature. We first prove it for many $\cat$ groups, including all nonelementary rank one $\cat$ groups, and all groups acting properly, minimally, nonelementarily by isometries on locally finite thick affine buildings. We also prove it for most hierarchically hyperbolic groups in the sense of Behrstock, Hagen and Sisto \cite{BHS1,BHS} -- see Theorem~\ref{theointro:hhs} below. These include mapping class groups of finite-type surfaces, thus answering a question raised by Boutonnet, Ioana and Peterson in \cite{BIP} -- see Theorem~\ref{theointro:mcg} below. The groups we consider are however not always bi-exact, because they may contain elements with nonamenable centralizer.

\paragraph*{Proper proximality among $\cat$ groups.} We recall that an isometry of a $\cat$ space $X$ is \emph{rank one} if it has an invariant axis that does not bound any half-space of $X$. An isometric action of a group $G$ on a $\cat$ space $X$ is \emph{nonelementary} if $G$ does not fix any point in $X$ and does not have any finite orbit in the visual boundary $\partial_\infty X$. Our first main result is the following.

\begin{theointro}\label{theointro:cat}
Let $G$ be a countable group acting properly nonelementarily by isometries on a proper $\cat$ space with a rank one element. Then $G$ is properly proximal.
\end{theointro}

In fact, we also prove the following more general version.

\begin{theointro}\label{theointro:cat-2}
Let $k\in\mathbb{N}$, and for every $i\in\{1,\dots,k\}$, let $X_i$ be a proper $\cat$ space, and let $G_i$ be a group acting by isometries on $X_i$. Let $G$ be a countable subgroup of $G_1\times\dots\times G_k$ which acts properly on $X_1\times\dots\times X_k$. Assume that for every $i\in\{1,\dots,k\}$, the projection $\pi_i(G)$ acts nonelementarily on $X_i$ with a rank one element. 

Then $G$ is properly proximal.
\end{theointro}

Note that proper $\cat$ spaces often admit a canonical decomposition as a product where the factors are either Euclidean or irreducible \cite[Theorem~5.1]{CM}, and the above theorem applies to such decompositions.

Notice also that we do not impose the properness of the $G_i$-action on $X_i$ in the above statement, but we only require the properness of the $G$-action on the product. It applies for example to some groups acting on non-affine buildings \cite{CapraceFujiwara} (in particular, Kac--Moody groups on finite fields are a class of infinite simple groups with property (T) satisfying this property \cite{CapraceRemy}). 

Combining Theorem~\ref{theointro:cat-2} with work of Caprace and Sageev \cite{CS} -- stating roughly that every proper group action on a proper finite-dimensional $\cat$ cube complex has a rank one element unless the space decomposes as a product, we reach the following corollary.

\begin{corintro}\label{corintro:ccc}
Let $G$ be a countable group acting properly nonelementarily by cubical automorphisms on a proper finite-dimensional $\cat$ cube complex. Then $G$ is properly proximal.
\end{corintro}

Our second main theorem in the $\cat$ setting is the following, where we recall that an action of a group $G$ on a $\cat$ space $X$ is \emph{minimal} if $X$ does not contain any strict nonempty $G$-invariant convex subspace.

\begin{theointro}\label{theointro:building}
Let $G$ be a countable group acting properly, minimally, nonelementarily by isometries on a locally finite thick affine building. Then $G$ is properly proximal.
\end{theointro}

 We note that affine buildings of higher rank are almost all classified by the famous work of Tits (see \cite{Weiss} for the full classification). Namely, if the dimension of the building is at least 3, then the building comes from an explicitely described algebraic construction, involving some non-archimedean field. In most cases this will imply that the isometry group of the building is a semisimple algebraic group over some local field, and therefore Theorem \ref{theointro:building} follows from the work of \cite{BIP}. However there are some cases when this is not the case: even in higher dimension, there are some constructions which are associated to groups which are not algebraic (e.g.\ involving vector spaces of infinite dimension). And more interestingly, the affine buildings of dimension 2 do not admit such a classification, and there are now some exotic constructions (see for example \cite{Tits_Andrews,Ronan_triangle,CMSZ2,Witzel}) which admit non-linear automorphism groups \cite{BCL}. 

A natural question that arises from our work is whether proper proximality holds for more general classes of groups with proper actions on $\mathrm{CAT}(0)$ spaces. The Rank Rigidity Conjecture predicts that if $G$ is a group acting properly and cocompactly on a proper geodesically complete CAT(0) space $X$, then either $G$ contains a rank one element, or $X$ splits as a product, or $X$ is a building or a symmetric space. In fact, the Rank Rigidity Conjecture is known to hold for $2$-dimensional piecewise Euclidean complexes \cite{BB} and for Hadamard manifolds \cite{Bal2,BS}, and in these situations this can be used to prove the following corollary.

\begin{corintro}\label{corintro:dim-2}
Let $G$ be a countable group acting properly cocompactly nonelementarily by isometries on either
\begin{enumerate}
\item a simply connected complete Riemannian manifold of nonpositive sectional curvature with no Euclidean factor, or
\item a $2$-dimensional piecewise Euclidean $\cat$ simplicial complex. 
\end{enumerate}
Then $G$ is properly proximal.
\end{corintro}

\paragraph*{Mapping class groups and hierarchically hyperbolic groups.} Hierarchically hyperbolic spaces and groups were introduced by Behrstock, Hagen and Sisto in \cite{BHS1} (with a streamlined definition in \cite{BHS}), to provide a common axiomatic framework for studying $\cat$ cubical groups and mapping class groups of finite-type surfaces. This framework is inspired from the Masur--Minsky hierarchy machinery from the surface setting \cite{MM}. The class of hierarchically hyperbolic groups includes various other examples, such as many $3$-manifold groups \cite[Theorem~10.1]{BHS}. Our second main result is the following, for which we refer to the opening paragraph of Section~\ref{sec:hhs} for relevant definitions. 

\begin{theointro}\label{theointro:hhs}
Let $(\calx,\mathfrak{S})$ be a hierarchically hyperbolic space, with $\calx$ proper.
 Let $G$ be a countable group acting nonelementarily by HHS automorphisms on $(\calx,\mathfrak{S})$, so that the $G$-action on $\calx$ is proper, by uniform quasi-isometries, and has an almost equivariant principal projection. 

Then $G$ is properly proximal.
\end{theointro}

The nonelementarity assumption has to be understood in the following way: the action of $G$ on the main hyperbolic space coming from the hierarchical structure is nonelementary (i.e.\ contains two independent loxodromic isometries). The most crucial example is the following, see \cite[Theorem~11.1]{BHS}. Let $g,n\in\mathbb{N}$, and let $\Sigma$ be a surface obtained from a closed, connected, orientable surface of genus $g$ by removing $n$ points. Then the mapping class group $\Mod(\Sigma)$ is properly proximal, unless $\Sigma$ is a sphere with no more than three punctures in which case $\Mod(\Sigma)$ is finite. This answers a question asked by Boutonnet, Ioana and Peterson in \cite{BIP}.

More generally, we establish the following theorem, for which we recall that the  \emph{FC-center} of a group $H$ is defined as the subgroup of $H$ made of all elements that centralize a finite-index subgroup of $H$.

\begin{theointro}\label{theointro:mcg}
Let $g,n\in\mathbb{N}$, let $\Sigma$ be a connected oriented surface of genus $g$ with $n$ points removed. Then every subgroup of $\Mod(\Sigma)$ which acts nonelementarily (i.e.\ with two independent loxodromic isometries) on the curve graph $\calc(\Sigma)$ is properly proximal.

More generally, an infinite subgroup $H\subseteq\Mod(\Sigma)$ is properly proximal if and only if its  FC-center is finite. 
\end{theointro}

In fact the last conclusion of this theorem also holds for surfaces with boundary, but in this case the mapping class group itself is never properly proximal because its weak center contains all peripheral Dehn twists -- however a subgroup $H$ avoiding the peripheral twists can be properly proximal. 

We would like to point out that the groups considered in Theorem~\ref{theointro:mcg} need not be hierarchically hyperbolic. Theorem~\ref{theointro:mcg} actually shows the proper proximality of many interesting subgroups of $\Mod(\Sigma)$: let us mention for example the Torelli subgroup or any subgroup from the Johnson filtration, or the \emph{handlebody group} -- i.e.\ the mapping class group of a $3$-dimensional handlebody, which naturally embeds as a subgroup of the mapping class group of the boundary surface.

\paragraph*{A word on the proofs.}

Our proofs for rank one $\cat$ groups and hierarchically hyperbolic groups rely on a general dynamical criterion, which is a variation over north-south dynamics, from which proper proximality follows. We let $G$ act on a compact space $\overline{X}=X\cup\partial X$ -- the visual compactification of the proper $\cat$ space in the rank one $\cat$ setting, or the compactification of the proper hierarchically hyperbolic space introduced by Durham, Hagen and Sisto in \cite{DHS}. The goal is then to find a $G$-invariant Borel subset $K^*$ inside $\overline{X}$ that contains a Cantor set, and such that every sequence $(g_n)_{n\in\mathbb{N}}\in G^\mathbb{N}$ has a subsequence $(g_{\sigma(n)})_{n\in\mathbb{N}}$ and an attracting point $\xi^+\in\partial X$ such that for all but one points $\xi\in K^*$, the sequence $(g_{\sigma(n)}\xi)_{n\in\mathbb{N}}$ converges to $\xi^+$.   

In the rank one $\cat$ setting, this dynamical criterion is checked by letting $K^*$ be the subspace of the visual boundary made of Morse geodesic rays, and applying a theorem of Papasoglu and Swenson \cite[Theorem~4]{PS}. In the hierarchically hyperbolic setting, we choose $K^*$ to be a Cantor set in the Gromov boundary of the main hyperbolic space coming from the hierarchical structure, which embeds as a subspace of the boundary built by Durham, Hagen and Sisto in \cite{DHS} -- see Proposition~\ref{prop:dynamics-boundary-hhs} for the verification of the dynamical criterion.

In the case of buildings, we use a slightly different criterion. As in the proof of \cite[Proposition~4.14]{BIP}, we use several compact spaces instead of just one. The boundary at infinity has a natural structure of a spherical building, and we define the compact spaces $K_i$ as the sets of simplices of a given type in $\partial_\infty X$. Each of these compact sets is endowed with a natural probability measure, and we prove that for every sequence $(g_n)_{n\in\mathbb{N}}\in G^\mathbb{N}$ there is a subsequence $(g_{\sigma(n)})_{n\in\mathbb{N}}$, an integer $i$ and an attracting point $\xi^+\in K_i$ such that for almost every point $\xi\in K_i$, the sequence $(g_{\sigma(n)}\xi)_{n\in\mathbb{N}}$ converges to $\xi^+$.  

We would like to point out that in the mapping class group setting, we doubt that other compact spaces such as the Thurston compactification of the Teichmüller space, or the space of all complete geodesic laminations with the Hausdorff topology, satisfy the required dynamical criterion. The boundary coming from the hierarchical structure is the only one with which we managed to run the argument. It is in fact the first example of a compactification of $\Mod(\Sigma)$ for which the group action on itself by left multiplication extends to a continuous action on the boundary, which is useful to get good dynamical properties.

\paragraph*{Applications to rigidity of von Neumann algebras.}

As already mentioned, our main motivation for proving proper proximality is to derive new rigidity results for certain von Neumann algebras, collected in Section~\ref{sec:von-Neumann} of the present paper. 

By the main theorem of \cite{BIP}, if $G$ is a properly proximal group, then the group von Neumann algebra $LG$ contains no weakly compact Cartan subalgebra. If $G\actson (X,\mu)$ is an ergodic measure-preserving essentially free $G$-action on a standard probability space $(X,\mu)$, then the group measure space von Neumann algebra $L^\infty(X)\rtimes G$ contains a weakly compact Cartan subalgebra if and only if the action $G\actson (X,\mu)$ is weakly compact (e.g.\ profinite), and in this case $L^\infty(X)$ is the unique weakly compact Cartan subalgebra up to unitary conjugacy. In particular, when $G\actson (X,\mu)$ is weakly compact, then $L^\infty(X)\rtimes G$ retains the orbit equivalence class of the action $G\actson (X,\mu)$ among weakly compact actions of countable groups. All these statements apply to all groups that arise in our main results.

For $\cat$ cubical groups, combining Corollary~\ref{corintro:ccc} with weak amenability (established by Guentner and Higson in \cite{GH}), we recover a theorem that follows from works of Ozawa and Popa \cite{OP} and Popa and Vaes \cite{PV}: all the above statements hold without the restriction that the Cartan subalegbras and the group actions be weakly compact (see Corollary~\ref{cor:von-neumann} for a precise statement).

Finally, in the mapping class group setting, combining proper proximality with a theorem of Kida \cite{Kid,Kid2} on orbit equivalence rigidity for ergodic probability measure-preserving actions of $\Mod(\Sigma)$, we derive that the von Neumann algebra $L^\infty(X)\rtimes\Mod(\Sigma)$ associated to a weakly compact ergodic probability measure-preserving essentially free action of $\Mod(\Sigma)$ recovers the action (up to stable conjugation) among the class of weakly compact actions of countable groups (see Theorem~\ref{theo:kida} for a precise statement). 

\paragraph*{Perspectives.} Rank one $\cat$ groups and hierarchically hyperbolic groups are two instances of non-positively curved groups in geometric group theory. It is natural to ask whether proper proximality holds for further classes of groups with features of non-positive curvature. One could e.g.\ wonder about groups acting properly on proper coarse median spaces or on spaces with convex geodesic bicombings, about systolic groups (where \cite{OP} provides a natural compactification to work with), automatic or biautomatic groups, or small cancellation groups for example.

Also, the question asked by Boutonnet, Ioana and Peterson, of whether $\Out(F_N)$ is properly proximal, remains open. The difficulty is that $\Out(F_N)$ is not hierarchically hyperbolic since it has exponential Dehn function. Finding a weak analogue of the hierarchy machinery of Masur and Minsky for $\Out(F_N)$ seems to be amongst the most challenging questions in the field; hopefully progress in this direction could yield a strategy to tackle the proper proximality question.

\paragraph*{Organization of the paper.} In Section~\ref{sec:proper-proximality}, we review the definition of proper proximality and establish a few dynamical criteria to check it. In Section~\ref{sec:cat}, we prove our proper proximality results among the class of $\cat$ groups. In Section~\ref{sec:hhs}, we prove Theorem~\ref{theointro:hhs} about hierarchically hyperbolic groups and derive Theorem~\ref{theointro:mcg} for mapping class groups and their subgroups. Finally, Section~\ref{sec:von-Neumann} collects all applications to rigidity of von Neumann algebras.

\paragraph*{Acknowledgments.}

We warmly thank Rémi Boutonnet for many discussions regarding proper proximality and rigidity of von Neumann algebras, and for very useful comments he made on an earlier version of this paper. We would also like to thank Cyril Houdayer for related discussions, and Bertrand Rémy and Bartosz Trojan for discussions about affine buildings.

The first named author acknowledges support from the Agence Nationale de la Recherche under Grant ANR-16-CE40-0006 DAGGER. The third named author acknowledges support from the Agence Nationale de la Recherche under Grants ANR-14-CE25-0004 GAMME and ANR-16-CE40-0022-01 AGIRA.

\setcounter{tocdepth}{2}
\tableofcontents

\section{Proper proximality of groups and group actions}\label{sec:proper-proximality}

In this section, we review the notion of proper proximality of a countable group introduced by Boutonnet, Ioana and Peterson in \cite{BIP}, and establish a few simple dynamical criteria to check it.

\subsection{Proper proximality of a countable group}

This section reviews work of Boutonnet, Ioana and Peterson  \cite{BIP}.

Let $G$ be a countable group. We denote by $\calp(G)$ the set of all subsets of $G$. We recall that an \emph{ultrafilter} on $G$ is a map $\omega:\mathcal{P}(G)\to\{0,1\}$ such that $\omega(G)=1$ and for all $A,B\subseteq G$, one has $\omega(A\cup B)=\omega(A)+\omega(B)-\omega(A\cap B)$. An ultrafilter $\omega$ on $G$ is \emph{nonprincipal} if for every finite subset $F\subseteq G$, one has $\omega(F)=0$. 

If $X$ is a topological space with a $G$-action, and $x,y\in X$, recall that $\lim\limits_{g\to \omega} gx=y$ means that for every open neighborhood $V$ of $y$ one has $\{ g\in G \mid gx\in V\}\in \omega$. 

Let now $K$ be a compact space equipped with a $G$-action, and let $\eta$ be a \emph{diffuse} probability measure on $K$, i.e.\ $\eta$ has no atom. Following \cite[Definition~3.7]{BIP}, we say that a nonprincipal ultrafilter $\omega$ on $G$ is \emph{$\eta$-proximal} if for every $h\in G$, one has $$\lim\limits_{g\to\omega}((gh)\cdot\eta-g\cdot\eta)=0$$ in the weak-$*$ topology. 

\begin{de}[{Boutonnet--Ioana--Peterson \cite[Definition~4.1]{BIP}}]
A countable group $G$ is \emph{properly proximal} if there exist finitely many compact $G$-spaces $K_1,\dots, K_\ell$, and for every $i\in\{1,\dots,\ell\}$, a diffuse probability measure $\eta_i$ on $K_i$, such that
\begin{enumerate}
\item for every $i\in\{1,\dots,\ell\}$, there is no $G$-invariant probability measure on $K_i$, and
\item for every nonprincipal ultrafilter $\omega$ on $G$, there exists $i\in\{1,\dots,\ell\}$ such that $\omega$ is $\eta_i$-proximal.
\end{enumerate}
\end{de}

 \subsection{Proper proximality of a group action}

Given a $G$-action by isometries on a proper metric space $X$, we say that a sequence $(g_n)_{n\in\mathbb{N}}\in G^\mathbb{N}$ \emph{escapes every compact subspace of $X$} if for some $x\in X$ (equivalently, for any $x\in X$), the sequence $(g_nx)_{n\in\mathbb{N}}$ escapes every compact subspace of $X$. The following definition is a variation on the definition of proper proximality in the context of group actions. 

\begin{de}\label{de:proximal-at-infinity}
Let $G$ be a group, and let $X$ be a proper metric space equipped with a $G$-action by isometries. We say that the $G$-action on $X$ is \emph{properly proximal} if there exist 
\begin{enumerate}
\item finitely many compact metrizable $G$-spaces $K_1,\dots,K_\ell$, none of which carries a $G$-invariant probability measure, and 
\item for every $i\in\{1,\dots,\ell\}$, a diffuse probability measure $\eta_i$ on $K_i$, 
\end{enumerate}
such that for every sequence $(g_n)_{n\in\mathbb{N}}\in G^\mathbb{N}$ that escapes every compact subspace of $X$, there exist $i\in\{1,\dots,\ell\}$ and a subsequence $(g_{\sigma(n)})_{n\in\mathbb{N}}$ such that for every $h\in G$, we have $\lim\limits_{n\to +\infty} (g_{\sigma(n)}h \eta_i-g_{\sigma(n)}\eta_i)=0$ in the weak-$*$ topology. 
\end{de}

\begin{lemma}\label{lemma:proximality}
Let $G$ be a countable group, and let $X$ be a proper metric space equipped with a $G$-action by isometries. If the $G$-action on $X$ is proper and properly proximal, then $G$ is properly proximal. 
\end{lemma}

\begin{proof}
Let $K_1,\dots,K_\ell$ be a finite collection of compact metrizable $G$-spaces witnessing the fact that the $G$-action on $X$ is properly proximal, as in Definition~\ref{de:proximal-at-infinity}, and let $\eta_{i}$ be the corresponding diffuse probability measures.

Let $\omega$ be a nonprincipal ultrafilter on $G$. For every $i\in\{1,\dots,\ell\}$, the space of all probability measures on $K_{i}$, equipped with the weak-$*$ topology, is compact. Therefore, for every $h\in G$, there exists a probability measure $\eta_{i,h}$ on $K_{i}$ such that $\lim\limits_{g\to \omega} gh \eta_{i}=\eta_{i,h}$. Let $F\subseteq G$ be a finite subset which contains the identity element. For each $h\in F$, we pick an arbitrary open neighborhood $V_{i,h}$ of $\eta_{i,h}$. Since $\omega$ is nonprincipal, it follows that the set $$\{g\in G\mid \forall i\in\{1,\dots,\ell\},~\forall h\in F,~ gh\eta_{i} \in V_{i,h}\}$$ is always infinite. As the space of probability measures on $K_i$ is metrizable, we can let the neighborhoods $V_{i,h}$ vary to be smaller and smaller balls around $\eta_{i,h}$, and deduce that there exists a sequence $(g_n)_{n\in\NN}\in G^\mathbb{N}$ of pairwise distinct elements such that for every $i\in\{1,\dots,\ell\}$ and every $h\in F$, we have $\lim\limits_{n\to +\infty} g_n h\eta_i=\eta_{i,h}$. 

As the $G$-action on $X$ is proper, there exists a subsequence $(g_{\sigma(n)})_{n\in \mathbb{N}}$ which escapes every compact subspace of $X$. As the $G$-action on $X$ is properly proximal, up to passing to a further subsequence, there exists $i\in\{1,\dots,\ell\}$ such that for every $h\in F$, we have $\lim\limits_{n\to +\infty}(g_{\sigma(n)} h\eta_{i}-g_{\sigma(n)}\eta_i)=0$, and therefore $\eta_{i,h}=\eta_{i,\mathrm{id}}$.  

We have thus proved that for every finite subset $F\subseteq G$, there exists $i\in\{1,\dots,\ell\}$ such that all $\eta_{i,h}$ with $h\in F$ are equal. It follows that there exists $i\in\{1,\dots,\ell\}$ such that all $\eta_{i,h}$ with $h\in G$ are equal. This shows that $\omega$ is $\eta_i$-proximal.
\end{proof}

For future applications, we record the following observation regarding products of properly proximal actions -- an interesting situation will be when the $G_i$-actions on the spaces $X_i$ are not assumed to be proper, but the $G$-action on $X_1\times\dots\times X_k$ is proper, allowing to apply Lemma~\ref{lemma:proximality}.

\begin{lemma}\label{lemma:product}
Let $k\in\mathbb{N}$. For every $i\in\{1,\dots,k\}$, let $G_i$ be a countable group, and let $X_i$ be a proper metric space equipped with a $G_i$-action by isometries which is properly proximal. Let $G$ be a subgroup of $G_1\times\dots\times G_k$ whose projection to each factor is surjective.

Then the $G$-action on $X_1\times\dots\times X_k$ is properly proximal.
\qed
\end{lemma}

Lemmas~\ref{lemma:north-south} and~\ref{lemma:criterion-proximality2} below describe two geometric situations where one can prove that a $G$-action is properly proximal. Their proofs rely on the following observation.

\begin{lemma}[see e.g.\ {\cite[Lemma~8.3]{FLM}}]\label{lemma:measures}
Let $G$ be a countable group, and let $K$ be a compact metrizable $G$-space. Let $(g_n)_{n\in\mathbb{N}}\in G^\mathbb{N}$, let $\lambda$ be a probability measure on $K$, and let $A\subseteq K$ be a Borel subset such that for $\lambda$-almost every $x\in K$, every accumulation point of the sequence $(g_n x)_{n\in\mathbb{N}}$ belongs to $A$. 

Then for every weak-$*$ accumulation point $\nu$ of $(g_n\lambda)_{n\in\mathbb{N}}$, one has $\nu(\overline{A})=1$.
\end{lemma}

A first simple geometric situation in which one can check the proper proximality of a group action is when one finds a compact metrizable space $K$ and a $G$-invariant subset $K^*\subseteq K$ which contains a Cantor set, where the action of $G$ behaves like north-south dynamics; more precisely, we have the following simple fact. 

\begin{lemma}\label{lemma:north-south}
Let $G$ be a countable group, and let $X$ be a proper metric space equipped with a $G$-action by isometries. Assume that there exist a compact metrizable $G$-space $K$ that does not carry any $G$-invariant probability measure, and a $G$-invariant Borel subset $K^*\subseteq K$ which contains a Borel subset homeomorphic to a Cantor set. 

Assume in addition that for every sequence $(g_n)_{n\in\mathbb{N}}\in G^\mathbb{N}$ that escapes every compact subspace of $X$, there exist a subsequence $(g_{\sigma(n)})_{n\in\mathbb{N}}$ and points $\xi^-,\xi^+\in K$ such that for all $\xi\in K^*\setminus\{\xi^-\}$, the sequence $(g_{\sigma(n)}\xi)_{n\in\mathbb{N}}$ converges to $\xi^+$.

Then the $G$-action on $X$ is properly proximal. 
\end{lemma}

\begin{proof}
As $K^*$ contains a Cantor subset, the space $K$ carries a diffuse probability measure $\eta$ such that $\eta(K^*)=1$. Let $(g_n)_{n\in\mathbb{N}}\in G^{\mathbb{N}}$ be a sequence that escapes every compact subspace of $X$, and let $(g_{\sigma(n)})_{n\in\mathbb{N}},\xi^-,\xi^+$ be as given by our assumption. As $K^*$ is $G$-invariant, for every $h\in G$, the measure $h\eta$ is again diffuse and satisfies $\eta(K^*)=1$ (and in particular $h\eta$ gives measure $0$ to the singleton $\{\xi^-\}$). It thus follows from Lemma~\ref{lemma:measures} that for every $h\in G$, one has $\lim\limits_{n\to +\infty}g_{\sigma(n)}h\eta=\delta_{\xi^+}$, concluding the proof. 
\end{proof}

\begin{rk}
We could have stated a version of Lemma~\ref{lemma:north-south} with finitely many compact spaces $K_i$ instead of just one, but this is the only form in which we will apply it in the sequel of the paper.
\end{rk}

The following lemma gives a second sufficient condition to check the proper proximality of a group action.

\begin{lemma}\label{lemma:criterion-proximality2}
Let $G$ be a countable group, and let $X$ be a proper metric space equipped with a $G$-action by isometries. Let $K_1,\dots,K_\ell$ be finitely many compact metrizable $G$-spaces, none of which carries a $G$-invariant probability measure. For every $i\in\{1,\dots,\ell\}$, let $\eta_i$ be a quasi-invariant probability measure on $K_i$.

Assume that for every sequence $(g_n)_{n\in\mathbb{N}}\in G^{\mathbb{N}}$ that escapes every compact subspace of $X$, there exist $i\in\{1,\dots,\ell\}$, a subsequence $(g_{\sigma(n)})_{n\in\mathbb{N}}$, and a point $\xi^+\in K_i$, such that for $\eta_i$-almost every $\xi\in K_i$, the sequence $(g_{\sigma(n)}\xi)_{n\in\mathbb{N}}$ converges to $\xi^+$. 

Then the $G$-action on $X$ is properly proximal.
\end{lemma}

\begin{proof}
Let  $(g_n)_{n\in\mathbb{N}}\in G^{\mathbb{N}}$ be a sequence that escapes every compact subspace of $X$. By assumption, there exist $i\in\{1,\dots,\ell\}$, a subsequence $(g_{\sigma(n)})_{n\in \mathbb{N}}$ and $\xi^+\in K$ such that for $\eta_i$-almost every $\xi\in K_i$, one has $\lim\limits_{n\to +\infty} g_{\sigma(n)} \xi=\xi^+$. Let $h\in G$. Since $\eta_i$ is quasi-invariant, this equality is also satisfied by $h_*\eta_i$-almost every $\xi\in K_i$. 
 It follows from Lemma~\ref{lemma:measures} that $\lim\limits_{n\to +\infty} g_{\sigma(n)}h \eta_i=\delta_{\xi^+}$, which concludes the proof.
\end{proof}

\section{Proper proximality among $\cat$ groups} \label{sec:cat}

\subsection{Proper proximality of rank one $\cat$ groups}\label{sec:rank-one}

Given a $\cat$ space $X$, we denote by $\partial_\infty X$ the visual boundary of $X$.  An isometric action of a group $G$ on a $\cat$ space $X$ is \emph{nonelementary} if $G$ does not fix any point in $X$ and does not have any finite orbit in $\partial_\infty X$. An isometry $g$ of $X$ is \emph{rank one} if there exists a $g$-invariant axis in $X$ which does not bound any half-plane. The goal of the present section is to establish the following theorem.

\begin{theo}\label{theo:rank-one}
Let $G$ be a countable group, and let $X$ be a proper $\cat$ metric space equipped with an isometric $G$-action. Assume that the $G$-action on $X$ is nonelementary and contains a rank one element.

Then the $G$-action on $X$ is properly proximal. In particular, if the $G$-action on $X$ is proper, then $G$ is properly proximal.
\end{theo}

Combined with Lemmas~\ref{lemma:proximality} and~\ref{lemma:product}, this also allows to give a version for product actions.

\begin{theo}\label{theo:main}
Let $k\in\mathbb{N}$. For every $i\in\{1,\dots,k\}$, let $X_i$ be a proper $\cat$ space, and let $G_i$ be a group acting by isometries on $X_i$. Let $G$ be a countable subgroup of $G_1\times\dots\times G_k$ which acts properly on $X_1\times\dots\times X_k$. Assume that for every $i\in\{1,\dots,k\}$, the projection $\pi_i(G)$ acts nonelementarily on $X_i$ with a rank one element.

Then $G$ is properly proximal.
\qed
\end{theo} 

Our strategy for proving Theorem~\ref{theo:rank-one} will be to apply the north-south-type dynamical criterion from Lemma~\ref{lemma:north-south}, choosing for $K$ the visual boundary $\partial_\infty X$, for $K^*$ the subspace made of Morse rays, and applying a theorem of Papasoglu and Swenson \cite[Theorem~4]{PS} to this setting. We start by constructing a Cantor set of Morse rays in the visual boundary.

\subsubsection{A Cantor set of Morse rays in the visual boundary}

Let $X$ be a $\cat$ space. Every rank one isometry of $X$ acts on $X\cup\partial_\infty X$ with north-south dynamics, see  e.g.\ \cite[Lemma~III.3.3]{Bal} or \cite[Lemma~4.4]{Ham}. Given a rank one element $g\in G$, we will denote by $g^{-\infty}$ and $g^{+\infty}$ its repelling and attracting fixed points in $\partial_\infty X$.

\begin{lemma}\label{lemma:free-group}
Let $X$ be a proper $\cat$ space, and let $G$ be a group acting nonelementarily by isometries on $X$ with a rank one element. Then $G$ contains a nonabelian free subgroup whose action on $X$ is proper, nonelementary and has a rank one element.
\end{lemma}

\begin{proof}
By \cite[Theorem~1.1]{Ham}, the group $G$ contains two rank one elements $g$ and $h$ with disjoint fixed point sets in $\partial_\infty X$. Let $U_g^+,U_g^-,U_h^+,U_h^-$ be pairwise disjoint open neighborhoods of $g^{+\infty},g^{-\infty},h^{+\infty},h^{-\infty}$ in $\overline{X}:=X\cup\partial_\infty X$, respectively, chosen so that there exists a point $x\in X\setminus (\overline{U_g^+}\cup \overline{U_g^-}\cup \overline{U_h^+}\cup \overline{U_h^-})$. Using north-south dynamics, we can find $n_0\in\mathbb{N}$ such that for every $n\ge n_0$, one has $$g^{n}(\overline{X}\setminus U_g^-)\subseteq U_g^+,~~~g^{-n}(\overline{X}\setminus U_g^+)\subseteq U_g^-,~~~h^{n}(\overline{X}\setminus U_h^-)\subseteq U_h^{+},~~~h^{-n}(\overline{X}\setminus U_h^+)\subseteq U_h^-.$$ By a ping-pong argument, the subgroup $H$ of $G$ generated by $g^{n_0}$ and $h^{n_0}$ is free and nonabelian. In addition, the group $H$ acts freely and discretely on the $H$-orbit of $x$, because every nontrivial element of $H$ sends $x$ into $U_g^+\cup U_g^-\cup U_h^+\cup U_h^-$ and the $H$-action on $H\cdot x$ is by isometries. 

Assume towards a contradiction that the $H$-action on $X$ is not proper. Then there exist $M\ge 0$, a point $y\in X$ and infinitely many pairwise distinct elements $h_n\in H$ such that for every $n\in\mathbb{N}$, we have $d(y,h_ny)\le M$. Using the triangle inequality, we deduce that for every $n\in\mathbb{N}$, we have $d(x,h_nx)\le M + 2d(x,y)$. As $X$ is proper, the closed ball of radius $M+2d(x,y)$ centered at $x$ is compact, so this contradicts the fact that the $H$-action on $H\cdot x$ is free and discrete.    
\end{proof}

A \emph{Morse gauge} is a function $N:[1,+\infty)\times [0,+\infty)\to [0,+\infty)$. Given a Morse gauge $N$, a quasigeodesic ray $r$ is \emph{$N$-Morse} if for every $K\ge 1$ and every $C\ge 0$, every $(K,C)$-quasigeodesic segment with both endpoints on $r$ stays contained in the closed $N(K,C)$-neighborhood of $r$. Given $L>0$, a quasigeodesic ray $r$ is \emph{$L$-strongly contracting} if for every ball $B$ disjoint from the image of $r$, the closest-point projection of $B$ to $r$ has diameter at most $L$.

If $N$ is a Morse gauge, let $\partial_M^N X_p$ be the set of equivalent classes of $N$-Morse geodesic rays emanating from a given base point $p$ (where two rays are equivalent if they stay at bounded Hausdorff distance from each other), equipped with the quotient topology of the topology of uniform convergence on compact sets. The \emph{Morse boundary} $\partial_M X_p$ is the union of the spaces $\partial_M^N X_p$ over all Morse gauges $N$. There is a natural partial ordering on the set of all Morse gauges (where $N\le N'$ if $N(K,C)\le N'(K,C)$ for all $K\ge 1$ and $C\ge 0$). The space $\partial_M X_p$ is equipped with the direct limit topology, see \cite{Cor}. It turns out that the space $\partial_M X_p$ and its topology does not depend on the choice of the base point $p$ \cite[Proposition~3.5]{Cor}, thus we will also write $\partial_M X$ instead of $\partial_M X_p$.

\begin{ex}
We will make use of the following two observations.
\begin{enumerate}
\item If $T$ is a tree, then there exists a Morse gauge $N$ such that every geodesic ray of $T$ is $N$-Morse, and therefore $\partial_M T=\partial_M^N T_p=\partial_\infty T$ for any choice of base point $p$, equipped with the usual topology.
\item 
If $X$ is a proper CAT(0) space, then for every $N$ and any base point $p\in X$, the set $\partial_M^N X_p$ is a subset of $\partial_\infty X$, and by definition the inclusion $\partial_M^N X_p\to \partial_\infty X$ is continuous.
\end{enumerate}
\end{ex}

\begin{prop}\label{prop:cantor-in-morse}
Let $X$ be a proper $\cat$ space which admits a nonelementary, isometric action of a group $G$ with a rank one element, and let $x_0\in X$. Then there exist a Morse gauge $N$ and a topological embedding of a Cantor set in $\partial_\infty X$ such that for every point $\xi$ in the image of this embedding, the geodesic ray $[x_0,\xi)$ is $N$-Morse.   
\end{prop}

\begin{proof}
Let $H$ be a subgroup of $G$ which acts properly, nonelementarily on $X$ with a rank one element: this exists by Lemma~\ref{lemma:free-group}. By \cite[Theorem~A]{Yan} (applied to the group $H$), there exist $K\ge 1, C\ge 0, L\ge 0$, an infinitely ended rooted simplicial tree $(T,\ast)$ -- with all edges assigned length $1$ -- and a quasi-isometric embedding $f:T\to X$ with $f(\ast)=x_0$, such that the $f$-image of every infinite ray in $T$ based at $\ast$ is an $L$-strongly contracting $(K,C)$-quasigeodesic ray in $X$ (the fact that $f$ is a quasi-isometric embedding and that the contraction constant is uniform over all rays of $T$ is not explicitly mentioned in the statement of \cite[Theorem~A]{Yan}, but it follows from the first paragraph of the proof of \cite[Lemma~3.2]{Yan}). By \cite[Lemma~3.3]{Sul}, there exists a Morse gauge $N'$ such that the $f$-image of every infinite ray in $T$ based at $\ast$ is $N'$-Morse. By \cite[Lemma~2.9]{Cor}, there exists a Morse gauge $N$ such that for every infinite ray $r$ in $T$, the unique geodesic ray in $X$ with the same origin and same endpoint in $\partial_\infty X$ as $f(r)$ is $N$-Morse. It follows that the map $f:T\to X$ induces a map $\partial_M f:\partial_M T\to\partial_M X$ from the Morse boundary of $T$ to that of $X$ which is Morse-preserving in the sense of \cite[Definition~4.1]{Cor}. Therefore by \cite[Proposition~4.2]{Cor}, the map $\partial_M f$ is a topological embedding from $\partial_M T$ to the subspace of the Morse boundary made of $N$-Morse rays. The Morse boundary $\partial_M T$ is homeomorphic to its visual boundary, a Cantor set. Since the topology on the Morse boundary of $X$ coincides with the topology on the visual boundary in restriction to $N$-Morse rays (with $N$ fixed), it follows that the image of $\partial_\infty T$ in $\partial_\infty X$ is a topologically embedded Cantor set.
\end{proof}

\subsubsection{Proper proximality of rank one $\cat$ groups}
 
In this section, we will complete our proof of Theorem~\ref{theo:rank-one}.  

\begin{lemma}\label{lemma:no-invariant-proba}
Let $X$ be a proper $\cat$ space, and let $G$ be a group acting by isometries on $X$. Assume that the $G$-action on $X$ is nonelementary and that $G$ contains a rank one element.

Then there is no $G$-invariant probability measure on $\partial_\infty X$.
\end{lemma}

\begin{proof}
Let $g\in G$ be a rank one element. Using Lemma~\ref{lemma:measures}, for every probability measure $\nu$ on $\partial_\infty X$, every weak-$*$ accumulation point of the sequence $(g^n\nu)_{n\in\mathbb{N}}$ is supported on $\{g^{-\infty},g^{+\infty}\}$. In particular, every $G$-invariant probability measure on $\partial_\infty X$ is supported on $\{g^{-\infty},g^{+\infty}\}$. But \cite[Theorem~1.1]{Ham} ensures that $G$ contains two rank one isometries with disjoint fix sets in $\partial_\infty X$, so the lemma follows.  
\end{proof}

Given a $\cat$ space $X$, we denote by $(\partial_\infty X)^\vis$ the subspace of $\partial_\infty X$ made of all \emph{visibility points}, that is, points $\xi\in \partial_\infty X$ such that every $\eta\in \partial_\infty X$ can be joined to $\xi$ by a geodesic in $X$.

\begin{lemma}\label{lemma:existence-diffuse-proba}
Let $X$ be a proper $\cat$ space, and let $G$ be a countable group acting by isometries on $X$. Assume that the $G$-action on $X$ is nonelementary and that $G$ contains a rank one element.

Then $(\partial_\infty X)^\vis$ contains a $G$-invariant Borel subset $K^*$ which contains a Borel subset homeomorphic to a Cantor set. 
\end{lemma}

\begin{proof}
A theorem of Charney and Sultan \cite{CS} ensures that every Morse geodesic ray is contained in $(\partial_\infty X)^\vis$. Proposition~\ref{prop:cantor-in-morse} thus ensures that $(\partial_\infty X)^\vis$ contains a Borel subset homeomorphic to a Cantor set, and the union of its $G$-translates forms the desired space $K^*$.
\end{proof}

The following statement about the dynamics of $G$ on $\partial_\infty X$ is a consequence of a theorem of Papasoglu and Swenson \cite{PS}. 

\begin{lemma}\label{lemma:ps}
Let $X$ be a proper $\cat$ space, and let $G$ be a group acting by isometries on $X$. Let $(g_n)_{n\in\mathbb{N}}\in G^\mathbb{N}$ be a sequence that escapes every compact subspace of $X$. Then there exist a subsequence $(g_{\sigma(n)})_{n\in\mathbb{N}}$ and $\xi^-,\xi^+\in\partial_\infty X$ such that for every $\xi\in (\partial_\infty X)^\vis\setminus\{\xi^-\}$, the sequence $(g_{\sigma(n)}\xi)_{n\in\mathbb{N}}$ converges to $\xi^+$.
\end{lemma}

\begin{proof}
Let  $\xi \in (\partial_\infty X)^\vis$ and $\xi'\in \partial_\infty X$. By definition $\xi$ and $\xi'$ are joined by a geodesic in $X$. Furthermore this geodesic cannot bound a half-plane, as otherwise the boundary of such a half-plane would contain points not joined to $\xi$.
Denoting by $d_T$ the Tits metric on $\partial_\infty X$, it thus follows from \cite[Proposition 9.21]{BH} that $d_T(\xi,\xi')>\pi$.
Hence a ball of radius $\pi$ in $\partial_\infty X$ (for the Tits metric) can contain at most one point from $(\partial_\infty X)^\vis$. The lemma then follows from \cite[Theorem~4]{PS} applied with $\theta=\pi$ (notice that \cite[Theorem~4]{PS} is stated for a proper $G$-action on $X$, but its proof -- and in particular the proof of \cite[Lemma~18]{PS} -- only requires the sequence $(g_{\sigma(n)})_{n\in\mathbb{N}}$ to escape every compact subspace of $X$).
\end{proof}

\begin{proof}[Proof of Theorem~\ref{theo:rank-one}] 
We check the criterion given in Lemma~\ref{lemma:north-south}, applied with $K=\partial_\infty X$ and with the $G$-invariant Borel subset $K^*$ contained in $(\partial_\infty X)^{\vis}$ and containing a Borel subset homeomorphic to a Cantor set given by Lemma~\ref{lemma:existence-diffuse-proba}. The fact that $K$ does not carry any $G$-invariant probability measure was checked in Lemma~\ref{lemma:no-invariant-proba}, and the main dynamical assumption from Lemma~\ref{lemma:north-south} was checked in Lemma~\ref{lemma:ps}. This completes our proof. 
\end{proof}

\subsection{Proper proximality of $\cat$ cubical groups}

Combining the main result of the previous section with the work of Caprace and Sageev \cite{CapraceSageev} establishing a Rank Rigidity Theorem in the realm of $\cat$ cube complexes, we establish the following theorem.

\begin{theo}\label{theo:ccc}
Let $G$ be a countable group acting properly nonelementarily by cubical automorphisms on a proper finite-dimensional $\cat$ cube complex $X$. 

Then there exist a $G$-invariant subcomplex $X'\subseteq X$ and a finite index subgroup $G^0\subseteq G$ such that the $G^0$-action on $X'$ is properly proximal. In particular $G$ is properly proximal.
\end{theo}

\begin{proof}
Notice that the second conclusion (proper proximality of $G$) follows from the first: indeed, the first conclusion together with Lemma~\ref{lemma:proximality} implies that $G^0$ is properly proximal, and proper proximality is stable under passing to a finite-index overgroup \cite[Proposition~1.6]{BIP}.
 
By \cite[Proposition~3.5]{CapraceSageev}, the space $X$ has a $G$-invariant subcomplex $X'$ on which the action is essential (in the sense of \cite[\S3.4]{CapraceSageev}) and still nonelementary. By \cite[Proposition~2.6]{CapraceSageev}, there is a decomposition $X'=X'_1\times\dots \times X'_p$ which is preserved by a finite index subgroup $G^0$ of $G$, and such that each $X'_i$ is irreducible. The action of $G^0$ on each $X'_i$ is again nonelementary, so by \cite[Theorem~6.3]{CapraceSageev} it contains a rank one isometry. By Theorem~\ref{theo:rank-one}, the $G^0$-action on each $X'_i$ is properly proximal, and Lemma~\ref{lemma:product} implies that the $G^0$-action on $X'$ is properly proximal. 
\end{proof}

\subsection{Proper proximality of groups acting on buildings}

Our goal in this section is to prove Theorem \ref{theointro:building}, namely, that groups acting properly, minimally, nonelementarily by isometries on locally finite thick affine buildings are properly proximal.

\subsubsection{Affine buildings and their boundaries}

We now review some facts about buildings that we will need in our proof, and refer to \cite{AbramenkoBrown} or \cite{Garrett} for more general information, and to \cite{KleinerLeeb} for a more metric point of view. 

Let $X$ be a  locally finite (simplicial) affine building. We assume furthermore that $X$ is thick (meaning that every codimension 1 simplex is contained in at least 3 different top dimensional simplices), and that $X$ does not split nontrivially as a product $T\times Y$ where $T$ is a tree and $Y$ another building.

The space $X$ is covered by \emph{apartments}, which are modeled on $\Sigma \simeq \mathbb R^n$ with a structure that we now describe. Fix an affine Coxeter group $W$, acting by affine isometries on $\Sigma$. Then $W$ can be written as a semidirect product $W_0\ltimes T$ where $W_0$ is a finite Coxeter group and $T$ is a group of translations of $\mathbb R^n$. We let $S$ be a Coxeter generating set for $W_0$ (made of simple reflections). The \emph{walls} of $\Sigma$ are the hyperplanes fixed by a conjugate of some $s\in S$. A \emph{halfspace} of $\Sigma$ is the closure of a connected component of $\Sigma$ with a wall removed.
The walls determine a simplicial structure on $\Sigma$ (in fact a polysimplicial structure if $X$ is a product), whose maximal simplices (called \emph{alcoves}) are the closures of the connected components of the complement of the walls. A \emph{special vertex} of $\Sigma$ is a vertex (for the above simplicial structure) whose stabilizer is a conjugate of $W_0$; equivalently, it is a vertex contained in a wall in every possible direction. Note that special vertices always exist. The \emph{type} of a vertex of $\Sigma$ is its $W$-orbit. Note that $W$ acts simply transitively on the set of alcoves of $\Sigma$, and that $T$ acts simply transitively on the set of vertices of a given type.

The simplicial structure on the building $X$ coincides with the above simplicial structure in restriction to every apartment of $X$. A fundamental axiom of buildings is that any two points are contained in a common apartment, and that there exists an isomorphism between any two apartments fixing their intersection. In particular, any two alcoves are contained in a common apartment (because any two points in their respective interiors are); in addition, for any two alcoves $c,c'\subseteq X$, we can define the \emph{Weyl distance} $\delta(c,c')$ as the unique element $w\in W$ such that $w$ sends $c$ to $c'$, when acting on some apartment containing both $c$ and $c'$. A fundamental theorem \cite[Theorem~5.73]{AbramenkoBrown}  in the theory of buildings is that a subset of $X$ which is $W$-isometric (meaning that there is a map which preserves the Weyl distance $\delta$) to a subset of $\Sigma$ is in fact contained in an apartment.  

The \emph{combinatorial convex hull} of two points $x,y\in \Sigma$, denoted $\Conv(x,y)$, is the intersection of all half-spaces containing both $x$ and $y$. If $x,y$ are points in $X$, then any apartment containing both $x$ and $y$ also contains $\Conv(x,y)$, so that we can talk of the combinatorial convex hull of $x$ and $y$ in a similar way. If $F\subseteq X$, then $\Conv(F)$ is the union of all $\Conv(x,y)$ for $x,y\in F$.

\medskip

A \emph{sector} based at a special vertex $x\in \Sigma$  is a subspace of $\Sigma$ equal to the closure of a connected component of the complement in $\Sigma$ of the union of all walls passing through $x$. It is a fundamental domain for the action of $W_0$ on $\Sigma$ (fixing $x$). More generally,  a \emph{sector-face} based at $x$ is the (unbounded) intersection of a sector with walls of $\Sigma$. 
Note that the visual boundary $\partial_\infty \Sigma$, with the Tits metric, is isometric to a Euclidean sphere. There is a simplicial structure on $\partial_\infty \Sigma$ where the visual boundaries of the sectors of $\Sigma$ are simplices  called \emph{Weyl chambers}; they are all isometric. The visual boundaries of the sector-faces are then the faces of the Weyl chambers. Again the \emph{type} of a sector-face, or of its visual boundary, is its $W$-orbit. The types of sector-faces are in bijection with the subsets of $S$ and each Weyl chamber has exactly one simplex of each type as a face. More generally, a simplex of type $I$ contains a simplex of type $J$ as a face if and only if $I\subseteq J$.

More generally, a \emph{sector} in $X$ is a subset simplicially isomorphic to a sector of $\Sigma$ (and therefore contained in an apartment by \cite[Theorem 11.53]{AbramenkoBrown}). A \emph{Weyl chamber} is the visual boundary of a sector; two sectors are \emph{equivalent} if they have the same visual boundary inside $\partial_\infty X$, or equivalently, if their intersection contains a sector. We write $\Delta$ for the set of Weyl chambers. If $x\in X$ is a  special vertex and $C$ is a Weyl chamber, then there is a unique sector based at $x$ with visual boundary $C$. We denote this sector $Q(x,C)$, it is the union of all geodesic rays between $x$ and a point in $C$.
 
 The set $\Delta$ is the set of Weyl chambers of some spherical building of Coxeter group $W_0$, which we describe now. For every $I\subseteq S$, a Weyl chamber $C$ contains exactly one simplex of type $I$ as a face. Denoting by $\Delta_I$ the set of simplices of type $I$, it follows that there is a natural map $\Delta\to \Delta_I$, and more generally, maps $\Delta_I\to \Delta_J$ whenever $I\subseteq J$. Furthermore, the boundary $\partial_\infty X$ endowed with the Tits metric is a metric realization of this simplicial complex \cite[Proposition~4.2.1]{KleinerLeeb}: each Weyl chamber is isometric to a model spherical simplex, and the distance between two points is the length distance in this piecewise spherical complex.
 
 The apartments of the spherical building $\Delta$ are precisely the boundaries of apartments of $X$. Any two Weyl chambers $C,C'\in\Delta$ are contained in the boundary of some apartment, and identifying this apartment with $\Sigma$ we see that there exists a unique element $w\in W_0$ such that $w C=C'$. This element $w$ does not depend on the choice of the apartment containing $C$ and $C'$ and is called the \emph{Weyl distance} between $C$ and $C'$, and denoted $\delta(C,C')$.
 
 The group $W_0$ contains a unique element of maximal  word length in $S$, which is denoted $w_0$. Two chambers $C,C'\in \Delta$ are \emph{opposite} if $\delta(C,C')=w_0$. Note that in $\Sigma$, two  Weyl chambers $C$ and $C'$ are opposite if and only if $C=-C'$. It follows that in $X$ two opposite Weyl chambers always contain points of $\partial_\infty X$ at Tits distance equal to $\pi$. More generally, a Weyl chamber $C$ will be \emph{opposite} to a simplex $\sigma\subseteq \partial_\infty X$ if there exists $C'$ adjacent to $\sigma$ and opposite to $C$; or equivalently, if for every point $\xi$ in $\sigma$, there exists a point $\xi'$ in $C$ at Tits distance $\pi$ from $\xi$.
 
 If $u\in \Delta_I$ is some simplex, then the \emph{residue}  $\Res(u)$ of $u$ is the set of Weyl chambers containing $u$. Since $X$ is supposed to be thick, the residue of every $u\in \Delta_I$ (for $I\neq \varnothing$) is uncountable. If $C\in \Delta$ and $u\in \Delta_I$, there is a unique $C'\in \Res(u)$ which is at minimal distance from $C$. We call it the \emph{projection} of $C$ on $\Res(u)$ and denote it by $\proj_u(C)$.
 
  Let $x\in X$ be a special vertex. For $z\in X$, let 
  $$\Omega_x(z)=\{C\in \Delta \mid z\in Q(x,C)\}\;.$$
    More generally, for a finite set $Z\subseteq X$ we define $\displaystyle\Omega_x(Z)=\cup_{z\in Z} \Omega_x(z)$.
    We define a topology on $\Delta$ by declaring that the sets $\Omega_x(z)$ are open. It turns out that for a fixed $x$, the sets $\Omega_x(z)$ form a basis of the topology on $\Delta$. The space $\Delta$ is then compact and metrizable, and a sequence $C_n$ converges to $C$ if for some (hence, all) special vertex $x\in X$ the sectors $Q(x,C_n)$ pointwise converge to $Q(x,C)$ in the sense that a finite subset of $X$ is contained in $Q(x,C)$ if and only if it is contained in $Q(x,C_n)$ for all sufficiently large $n\in\mathbb{N}$.
    We endow $\Delta_I$ with the quotient topology from the natural map $\Delta\to \Delta_I$. Again, a sequence $(\tau_n)_{n\in \NN}$ of simplices in $\Delta_I$ converges to $\tau$ if the sector-faces $Q(x,\tau_n)$ pointwise converge to $Q(x,\tau)$.
 
 \medskip

As we have seen, $\partial_\infty X$ with the Tits metric is a metric realization of the spherical building associated to $\Delta$. 
Let $\mathfrak c$ be a model simplex for the metric realization of a Weyl chamber, and $\mathfrak c_I$ be  its simplex of type $I\subseteq S$. Note that the only isometry of $\mathfrak c$ which preserves the type of simplices in $\partial \mathfrak c$ is the identity. 
Therefore, for any $C\in \Delta$  there is a unique type-preserving isometry $\mathfrak c\to C$
(viewing $C$ as a subset of $\partial_\infty X$) and this defines a map $R:\mathfrak c\times  \Delta \to \partial_\infty X$. More generally, for every $I\subseteq S$, we have a map $R_I:\mathfrak{c}_I\times\Delta_I\to\partial_\infty X$.

Conversely, for every $\xi$ in $\partial_\infty X$, there is a unique $\theta\in \mathfrak c$ such that there is a chamber $C\in \Delta$ with $R(\theta,C)=\xi$. In that case we simply write $\theta=\theta(\xi)$. Furthermore,  if $\xi$ is contained in the interior of a Weyl chamber, then $C$ is also unique, and we also write $C(\xi)$ for this $C$. More generally, $\xi$ is contained in the interior of some simplex of type $I$, which is unique and that we denote $\tau(\xi)$.
 
\begin{lemma}\label{lem:visualvscomb}
Endow $\partial_\infty X$ with the visual topology. Then for every $I\subseteq S$ and every $\theta$ in the interior of  $\mathfrak c_I$, the map $R_I(\theta,\cdot):\Delta_I\to \partial_\infty X$ is a homeomorphism onto its image, which is closed.
\end{lemma}

\begin{proof}
Fix $\theta$ in the interior of $\mathfrak c_I$. Let $(\tau_n)_{n\in \NN}\in \Delta_I^\NN$. Assume that $(\tau_n)_{n\in \NN}$ converges to some $\tau\in \Delta_I$. Let $\xi_n=R_I(\theta,\tau_n)$ and $\xi=R_I(\theta,\tau)$, and let $\rho_n$ (resp.\ $\rho$) be the geodesic ray starting at $x$ to $\xi_n$ (resp.\ to $\xi$) for some special vertex $x$. The convergence of $\tau_n$ to $\tau$ in $\Delta_I$ means that for every finite subset $F$ of $Q(x,\tau)$ we have for $n$ large enough $F\subseteq Q(x,\tau_n)$. In particular for $n$ large enough $Q(x,\tau_n)$ contains a large initial segment of $\rho$, and therefore $\xi_n$ converges to $\xi$. This proves the continuity of $R_I(\theta,\cdot)$.

Now let again $(\tau_n)_{n\in \NN}\in \Delta_I^\NN$ and $\xi_n=R_I(\theta,\tau_n)$. 
Assume that $\xi_n$ converges to some $\xi\in \partial_\infty X$. We will prove that $\tau(\xi)$ is of type $I$, and in fact $\xi$ can be written as $R_I(\theta,\tau)$ with $\tau\in\Delta_I$, and in addition $\tau_n$ converges to $\tau$: this will prove both that $R_I(\theta,\cdot)^{-1}$ is continuous and that $R_I(\theta,\Delta_I)$ is closed.  Let $\rho_n$ be the geodesic ray from $x$ to $\xi_n$.
 Since $\xi_n=R_I(\theta,\tau_n)$, the (open) simplex containing the initial segment of $\rho_n$ is always of the same dimension. As $(\xi_n)_{n\in\mathbb{N}}$ is a convergent sequence, this simplex must eventually be constant. 
  Let $A_n$ be an apartment containing $\rho_n$. Then $\partial_\infty A_n$ is isometric to the space of directions (in $A_n$) at $x$, and therefore $\theta$ determines the direction of $\rho_n$. Since furthermore the initial segment $\rho_n$ is (for $n$ large enough) in a constant simplex, it follows that this segment is eventually constant, 
 and therefore equal to the one of $\rho$. 

Write $\xi=R_I(\theta',\tau)$ for some $\tau\in \Delta_J$ with $J\subseteq S$ and $\theta'$ in the interior of $\mathfrak c_J$. Fix an apartment $A$ containing $\rho$. Using again that $\partial_\infty A$ is isometric to the space of directions at $x$, we see that both $J$ and $\theta'$ are determined by the germ of $\rho$ at $x$. Since for $n$ large enough $\rho_n$ and $\rho$ have the same germ at $x$ we deduce that $J=I$ and that $\theta'=\theta$, in other words that $\xi$ is of the form $R_I(\theta,\tau)$ for some $\tau \in \Delta_I$.

Using the same argument repeatedly we find that for $n$ large enough $\rho_n$ and $\rho$ have the same initial segment of arbitrary large length. 
Since $\theta$ is in the interior of $\mathfrak c_I$, for every boundary wall $H$ of the sector face $Q(x,\tau)$, we have $\lim\limits_{k\to +\infty} d(\rho(k),H)=+\infty$. It follows that for  every $y\in Q(x,\tau)$ we have for $k$ large enough $y\in \Conv(x,\rho(k))$. 
 Thus for $k$ large enough we have $F\subseteq \Conv(x,\rho(k))$ and therefore for $n$ large enough we have $F\subseteq Q(x,\tau_n)$.
 Hence $\tau_n$ converges to $\tau$, which completes the proof.
\end{proof}

\subsubsection{Construction of the measure}
We now wish to construct a probability measure on $\Delta$ (and consequently on the spaces $\Delta_I$) which will satisfy the proper proximality condition. 

First, let $\Sigma^+$ be a sector in $\Sigma$ based at $0$. For every pair of special vertices $x,y\in X$ there is a unique simplicial map $\phi:\Sigma\to X$, which preserves the type at infinity,
 such that $\phi(0)=x$ and $\phi(y)\in \Sigma^+$. The point $\phi(y)$ is called the \emph{vectorial distance} between $x$ and $y$ and denoted $\sigma(x,y)$. 
 For every $\lambda\in \Sigma^+$   there exists a unique $\lambda^*\in \Sigma^+$ such that for all pairs of special vertices $x,y\in X$ one has $\sigma(x,y)=\lambda$ if and only if $\sigma(y,x)=\lambda^*$.
 We define, for $x\in X$ and $\lambda\in \Sigma^+$,
$$V_\lambda(x)=\{y\in X\mid \sigma(x,y)=\lambda\}.$$

Note that the building $X$ is \emph{regular} in the terminology of \cite{Parkinson1} because it does not have any tree factor, see \cite[Theorem~2.4]{Parkinson1}. Therefore, it follows from \cite[Theorem~5.15]{Parkinson1} that the cardinality of $V_\lambda(x)$ is the same for all special vertices $x$; we denote it $N_\lambda$.

\begin{prop}
Let $o\in X$ be a special vertex. Then there is a unique probability measure $\mu_o$ on $\Delta$ such that $\mu_o(\Omega_o(z))=\frac{1}{N_{\sigma(o,z)}}$.

Furthermore, if $o'$ is another special vertex then $\mu_{o'}$ and $\mu_o$ have the same null sets. 
\end{prop}

The construction of the measure is explained in \cite[p.587]{Parkinson2}; the last part of the proposition is proved in \cite[Theorem~3.17]{Parkinson2}.

Let $C\in \Delta$, and let $A$ be an apartment such that $C\subseteq\partial_\infty A$. Recall that the \emph{retraction} on $A$ centered at $C$ is the unique map $\rho_{A,C}:X\to A$ such that for every apartment $A'$ that contains a sector in the class of $C$, the restriction of $\rho_{A,C}$ to $A'$ is the unique isomorphism $A'\to A$ fixing $A'\cap A$ pointwise. The following definition comes from \cite[Theorem 3.4]{Parkinson2}, where it is shown in particular not to depend on the choice of the apartment $A$.

\begin{de}
Let $C\in \Delta$, and let $x\in X$ be a special vertex. Let $A$ be an apartment containing $Q(x,C)$, and $\psi:A\to \Sigma$ be the unique simplicial isomorphism, preserving the type at infinity, such that $\psi(x)=0$ and $\psi(Q(x,C))=\Sigma^+$.

The \emph{vectorial horofunction} associated to $C$ is the function $h_{x,C}:X\to \Sigma$ defined as $h_{x,C}=\psi\circ \rho_{A,C}$.
\end{de}

The following proposition was proved in \cite[Proposition 6.8]{BCL}, in the case of buildings of type $\widetilde A_2$. Here we treat the general case.

\begin{prop}\label{prop:opposition-measure}
Let  $C_0\in \Delta$. Then $\mu_o$-almost every $C$ is opposite $C_0$.
\end{prop}

In order to prove Proposition \ref{prop:opposition-measure}, let us introduce some notations. Recall that $w_0$ is the longest element of $W_0$.
For every special vertex $x\in X$, let $\Delta_x(C_0)$ be the set of $C\in \Delta$ such that $Q(x,C)\cup Q(x,C_0)$ is contained in an apartment.  For every $w\in W_0$, we write $\Delta^w(C_0)=\{ C\in\Delta\mid \delta(C_0,C)=w\}$ and $\Delta_x^w(C_0)=\Delta_x(C_0)\cap \Delta^w(C_0)$. Finally, for $\lambda\in \Sigma^+$ we define 
$$Y_\lambda^w(x,C_0)=\{ y\in X\mid \sigma(x,y)=\lambda \textrm{ and } \Omega_x(y)\cap \Delta_x^w(C_0)\neq \varnothing\}.$$

For $\lambda,\nu\in \Sigma^+$, let $\Pi_\lambda$ be the convex hull of $W_0.\lambda$, and write $\lambda \ll \nu$ if $\nu-\Pi_\lambda\subseteq \Sigma^+$.

\begin{lemma}\label{lem:nulambda}
Let $x\in X$ be a special vertex, and let $w\in W_0$. Then $$Y_\lambda^w(x,C_0)=\{ y\in V_\lambda(x) \mid h_{x,C_0}(y)=w\lambda\}.$$
\end{lemma}

\begin{proof}
Let $y\in Y_{\lambda}^w(x,C_0)$. Then there exists $C_1\in \Omega_x(y)\cap \Delta_x^w(C_0)$. In particular, there exists an apartment $A$ containing $Q(x,C_0)\cup Q(x,C_1)$, and in this apartment we see from the definition of $h_{x,C_0}$ that $h_{x,C_0}(y)=w\lambda$.

Conversely, let $y\in  V_{\lambda}(x)$ be such that $h_{x,C_0}(y)=w\lambda$. Let $\lambda \ll\nu$ in $\Sigma^+$.  Let $z$ be the (unique) point in $V_\nu(x)\cap Q(x,C_0)$. We claim that $x,y$ and $z$ are in fact in a common apartment, and therefore that $\Conv(x,z)$ and $y$ are in an apartment. 
Since $\nu$ can be taken arbitrarily large, it will follow that for every finite subset $F\subseteq Q(x,C_0)$ there exists an apartment containing $F$ and $y$. By \cite[Theorem 5.73]{AbramenkoBrown}  it follows that $Q(x,C_0)\cup\{y\}$ is contained in an apartment. In this apartment, let $C_1$ be the Weyl chamber  such that $\delta(C_0,C_1)=w$.  Since $h_{x,C_0}(y)=w\lambda$
we have $y\in Q(x,C_1)$. Therefore $C_1\in \Omega_x(y)\cap \Delta_x^w(C_0)$ and it follows that $y\in Y^w_\lambda(C_0)$.

We now have to prove the claim.  By \cite[Theorem 3.6]{Parkinson2} we have $\sigma(y,z)=\nu-h_{x,C_0}(y)$ (which belongs to $\Sigma^+$ by \cite[Theorem 3.4]{Parkinson2} because $\lambda\ll\nu$). 

Let $c_0$ be an alcove containg $z$, and let $c_1$ and $c_2$ be alcoves containing $x$ and $y$, respectively, such that $d(c_1,c_2)$ is minimal. The Weyl distance  between $c_1$ and $c_2$ is then uniquely determined by the distance $\sigma(x,y)$ (see \cite[Lemma~5.36]{AbramenkoBrown}).
Let $A$ be an apartment containing $x$ and $z$, and $\rho:X\to A$ be the retraction to $A$ centered at $c_0$. By definition we have $\rho(c_0)=c_0$ and $\rho(c_1)=c_1$. Let $y'=\rho(y)$ and $c'_2=\rho(c_2)$. Then  $\sigma(y',z)=\sigma(y,z)=\nu-h_{x,C_0}(y)=\nu-w\lambda$ by definition of the retraction, and since $\rho$ is $1$-Lipschitz we have $d(x,y')\leq d(x,y)$, and in fact $\sigma(x,y')\in \Pi_\lambda$ \cite[Theorem 3.3]{Parkinson2}.

But as $\nu \gg \lambda$, it follows that the point $y'$ is the unique point in $A\cap V_{(\nu-w\lambda)^*}(z)$ such that $\sigma(x,y')\in \Pi_\lambda$. Therefore $y'$ is the point of $A$ such that $h_{x,C}(y')=w\lambda$, and we have in particular $\sigma(x,y')=\lambda=\sigma(x,y)$. Since $d(c'_2,c_1)\leq d(c_2,c_1)$ and $c'_2$ is adjacent to $y$ it follows from the previous remark that the Weyl distance between $c_1$ and $c'_2$ is the same as the Weyl distance between $c_1$ and $c_2$. Likewise the Weyl distance between $c_0$ and $c'_2$ is the same as the Weyl distance between $c_0$ and $c_2$.
  This means that $\rho_{|c_0,c_1,c_2}$ is a $W$-isometry. By \cite[Theorem 5.73]{AbramenkoBrown} it follows that $\{x,y,z\}$ is contained in an apartment. This proves the claim, and therefore the lemma.
\end{proof}

\begin{lemma}\label{lem:mememesure}
For every special vertex $x\in X$ and every $w\in W_0$, for every $C_0,C_1\in \Delta$, we have $\mu_x(\Delta_x^w(C_0))=\mu_x(\Delta_x^w(C_1))$.
\end{lemma}

\begin{proof}
Fix $C_0\in \Delta$ and $w\in W_0$. Let $(\lambda_n)_{n\in \NN}$ be a sequence of elements of $\Sigma^+$ such that $\lambda_n \ll \lambda_{n+1}$. Then $\Delta_x^w(C_0)=\bigcap_{n\in \NN} \Omega_x(Y_{\lambda_n}^w(x,C_0))$, and furthermore this is a decreasing intersection. Notice in addition that if $y,y'\in Y_{\lambda_n}^w(x,C_0)$ are distinct, then in particular $\sigma(x,y)=\sigma(x,y')$, and therefore $\Omega_x(y)$ and $\Omega_x(y')$ are disjoint. Hence we get $$\mu_x(\Delta_x^w(C_0))=\lim\limits_{n\to+\infty} \frac{ | Y_{\lambda_n}^w(x,C_0)|}{|V_{\lambda_n}(x)|}\;.$$

Therefore, it suffices to prove that for every $n$ the cardinality  $|Y_{\lambda_n}^w(x,C_0)|$ does not depend on $C_0$. 
By Lemma \ref{lem:nulambda}, we have $$Y_{\lambda_n}^w(x,C_0)= \{ y\in V_{\lambda_n}(x) \mid h_{x,C_0}(y)=w\lambda_n\},$$ and by 
 \cite[Lemma~3.19]{Parkinson2} this quantity does not depend on the choice of $C_0$ (or $x$). This concludes our proof.
\end{proof}

\begin{proof}[Proof of Proposition \ref{prop:opposition-measure}]
Let $V$ be the set of special vertices of $X$.
Note that $\Delta=\cup_{x\in V} \Delta_x(C_0)$. In particular for every $w$ we have $\Delta^w=\cup_{x\in V} \Delta_x^w(C_0)$. Therefore it suffices to prove that for every $w\neq w_0$ and every $x\in V$, we have $\mu_o(\Delta_x^w(C_0))=0$. Since $\mu_o$ and $\mu_x$ are equivalent it suffices to prove that for every $x\in V$, we have $\mu_x(\Delta_x^w(C_0))=0$.

So fix $x\in X$ and $w\in W$, with $w\neq w_0$. In particular there exists a simplex  $u$ of $\Delta$, not of maximal dimension, such that for every $C\in \Res(u)$ and $D\in \Delta^w(C)$ we have $\proj_u(D)=C$. Hence for $C,C'\in \Res(u)$ with $C\neq C'$ we have $\Delta^w(C)\cap \Delta^w(C')=\varnothing$, in particular $\Delta_x^w(C)\cap \Delta_x^w(C')=\varnothing$.  By Lemma \ref{lem:mememesure}, if we have $\mu_x(\Delta_x^w(C))=m>0$ for some $C$ it would follow that $\mu_x(\Delta_x^w(C))=m$ for every $C\in \Res(u)$. Since $\Res(u)$ is infinite (by thickness of $X$) and $\mu_x(\Delta)<+\infty$ this is impossible.  

Hence we have $\mu_x(\Delta_x^w(C))=0$ for every chamber $C\in \Delta$ and every $w\neq w_0$, which proves the proposition.
\end{proof}

\subsubsection{Proper proximality}

Let $G$ be a countable group. An isometric $G$-action on a building $X$ is \emph{minimal} if there is no proper closed convex invariant subset of $X$.

\begin{theo}\label{thm:buildings}
Let $G$ be a countable group which admits a proper minimal isometric action on a locally finite thick affine building $X$, with no finite orbit in $\partial_\infty X$.

Then $G$ is properly proximal.
\end{theo}

We will be using the following general result on groups acting on $\cat$ spaces.

\begin{theo}[Adams--Ballmann \cite{AdamsBallmann}]\label{theo:AB}
Let $G$ be a group, and let $X$ be a locally compact $\cat$ space equipped with an isometric $G$-action. Assume that the action of $G$ on $X$ is minimal, that $\partial_\infty X\neq\emptyset$, and that $\partial_\infty X$ supports a $G$-invariant probability measure $\lambda$. Then at least one of the following two assertions holds:
\begin{enumerate}
	\item $G$ fixes a point in $\partial_\infty X$;
	\item $X$ has a nontrivial Euclidean factor.
\end{enumerate}
\end{theo}

\begin{proof}
We follow \cite{AdamsBallmann}. Let $x\in X$. For $\xi\in \partial_\infty X$, let $b_\xi$ be the Busemann function associated to $\xi$ satisfying $b_\xi(x)=0$.  Let $b$ be the convex function $b=\int_{\partial_\infty X} b_\xi d\lambda(\xi)$. Suppose $G$ does not fix a point in $\partial_\infty X$, by \cite[Lemmas~2.4 and~2.5]{AdamsBallmann} we deduce that $b$ is $G$-invariant. In particular, every level set of $b$ is $G$-invariant, and by minimality it follows that the function $b$ is constant. By convexity it follows that for $\lambda$-almost every $\xi$ the function $b_\xi$ is both convex and concave, hence it is affine (meaning that it is an affine function in restriction to every geodesic). Hence by \cite[Theorem~1.6 and Lemma~1.7]{AdamsBallmann} we get that either $X$ has a non-trivial flat factor or $G$ fixes a point in $\partial_\infty X$.
\end{proof}

The following is an immediate consequence of Theorem~\ref{theo:AB}, as Lemma~\ref{lem:visualvscomb} implies that if $\Delta_I$ had an invariant probability measure, then $\partial_\infty X$ would also carry an invariant probability measure.

\begin{lemma}\label{lem:building-noinvmeasure}
Under the assumptions of Theorem \ref{thm:buildings}, 
the compact spaces $\Delta_I$ do not carry any $G$-invariant probability measure. 
\qed
\end{lemma}

\begin{proof}[Proof of Theorem \ref{thm:buildings}]
First, let us assume that $X$ does not have a factor which is a tree. We will prove that the $G$-action on $X$ is properly proximal, from which the proper proximality of $G$ follows by Lemma~\ref{lemma:proximality} as the $G$-action on $X$ is proper. For this, we will apply the criterion provided by Lemma~\ref{lemma:criterion-proximality2}. The finitely many compact spaces which will witness proximality are the spaces $\Delta_I$, for $I\subseteq S$. Fix some special vertex $o\in X$. Then the spaces $\Delta_I$  are equipped with the projection $\mu_I$ of the measure $\mu=\mu_o$.

Let $(g_n)_{n\in\NN}\in G^\NN$ be a sequence of pairwise distinct elements. By \cite[Theorem~4]{PS}, up to replacing $(g_n)_{n\in \NN}$ by a subsequence, we may and shall assume that there exist points $p,m\in \partial_\infty X$ such that for every $\theta\in (0,\pi)$, for every $\xi\in \partial_\infty X\setminus B_T(m,\pi-\theta)$, every limit point of $(g_n \xi)_{n\in \NN}$ is in $B_T(p,\theta)$.  Let $I\subseteq S$ and $\phi$ in the interior of $\mathfrak c_I$ such that $p=R_I(\phi,\sigma)$ for some $\sigma\in \Delta_I$.
Choose $\theta$ small enough such that $B_T(p,\theta)$ does not intersect any simplex not containing $p$.

Let $C_0\in \Delta$ be a Weyl chamber containing $m$, and let  $C\in \Delta$ be opposite to $C_0$. Then  $C$ contains some point $q$ with $d_T(m,q)=\pi$. If $q'\in \partial_\infty X$ is such that $d_T(q,q')<\theta$ then $d_T(m,q')>\pi-\theta$, and therefore every limit point of $(g_n q')_{n\in \NN}$ will be contained in $B_T(p,\theta)$. In particular, taking $q'$ in the interior of $C$, we see by Lemma~\ref{lem:visualvscomb} that any limit point of $(g_n C)_{n\in \NN}$ (in the space $\Delta$) is  a chamber that contains $\sigma$, and therefore if $\tau$ is the face of $C$ of type $I$ then $g_n \tau$ converges to $\sigma$. This is true for every chamber $C$ opposite to $C_0$, and therefore by Proposition \ref{prop:opposition-measure} for $\mu$-almost every chamber $C$. It follows that $\mu_I$-almost every simplex $\tau$ satisfies that  $g_n\tau$ converges to $\sigma$.

Since for every $g\in G$ we have $g_*\mu_o=\mu_{g o}$, which is equivalent to $\mu_o$, the measure $\mu_{o}$ is $G$-quasi-invariant.
By Lemma~\ref{lemma:criterion-proximality2}, it follows that the $G$-action on $X$ is properly proximal.

We now turn to the general case, where $X$ is allowed to have a tree factor. In general, we can write $X=Y\times Z$, where $Z$ is a product of trees and $Y$ is a building without a tree factor, and this decomposition is $G$-invariant, so that $G$ acts by isometries on each factor.
Since the $G$-action on $X$ is minimal, so are the $G$-actions on $Y$ and $Z$ -- in particular there are no fixed points for these actions. In addition there is no finite orbit in $\partial_\infty Y$ or $\partial_\infty Z$. By the above, the $G$-action on $Y$ is properly proximal, and as $Z$ naturally has the structure of a proper $\cat$ cube complex, it follows from Theorem~\ref{theo:ccc} that $G$ has a finite-index subgroup $G^0$ such that the $G^0$-action on $Z$ is properly proximal (notice that we do not have to pass to a further invariant subcomplex here as the $G$-action on $Z$ is minimal). Lemma~\ref{lemma:product} then implies that the $G^0$-action on $X=Y\times Z$ is properly proximal. As this action is proper, Lemma~\ref{lemma:proximality} implies that $G^0$ is properly proximal, and therefore so is $G$ by \cite[Proposition~1.6]{BIP}.
\end{proof}

\subsection{Two consequences of Rank Rigidity}

The Rank Rigidity Conjecture predicts that if a countable group $G$ acts properly cocompactly on a proper geodesically complete $\cat$ space $X$, then either $G$ is a symmetric space or a Euclidean building, or $G$ contains a rank one isometry, or $G$ splits as a direct product. Combined with Theorems~\ref{theo:main} and~\ref{thm:buildings} above, this would show that every such group $G$ is properly proximal. By work of Ballmann and Brin \cite{BB}, the Rank Rigidity Conjecture is known to hold for piecewise Euclidean $2$-dimensional simplicial complexes, which leads to the following theorem.

\begin{theo}
Let $G$ be a group acting properly cocompactly nonelementarily by isometries on a $2$-dimensional piecewise Euclidean $\cat$ simplicial complex $X$. Then $G$ is properly proximal.
\end{theo}

\begin{proof}
By a theorem of Ballmann and Brin \cite[Theorem~C]{BB}, one of the following three situations occur: 
\begin{enumerate}
\item $G$ contains a rank one isometry,
\item $X$ is a direct product of two trees,
\item $X$ is a thick Euclidean building.
\end{enumerate}
In the first situation, proper proximality of $G$ follows from Theorem~\ref{theo:main}. In the second situation $X$ has a natural structure of a $\cat$ cube complex, and proper proximality follows from Theorem~\ref{theo:ccc}. In the third situation, note that by \cite[Lemma~3.13]{CM} the action of $G$ on $X$ is minimal. The
proper proximality of $G$ follows from Theorem~\ref{thm:buildings}.
\end{proof}

The Rank Rigidity conjecture is also verified for groups acting on Riemannian manifolds (see \cite{Bal2} or \cite{BS}), leading to the following theorem.

\begin{theo}
Let $M$ be a complete, simply connected Riemannian manifold of non-positive sectional curvature without a flat factor. Let $G$ be a group acting properly cocompactly nonelementarily by isometries on $M$. Then $G$ is properly proximal.
\end{theo}

\begin{proof}
Up to replacing $M$ by a submanifold, we may and shall assume that $M$ does not contain any complete $G$-invariant nontrivial strict submanifold.
The de Rham decomposition \cite[Th\'eor\`eme~III]{deRham} yields a canonical decomposition of $M$ as a product of irreducible Riemannian manifolds $M=M_1\times \dots\times M_n$, and an isometry of $M$ restricts to an isometry of each $M_i$ (up to possibly permuting the isometric factors). Furthermore by assumption no $M_i$ is isometric to some Euclidean space. By \cite{Bal2}, each $M_i$ is then either a higher rank symmetric space of noncompact type, or a manifold of rank one. 

Replacing if necessary the group $G$ by a subgroup of finite index, we can assume that $G$ preserves each factor $M_i$, that is, we can write $G=G_1\times\dots\times G_n$ where each $G_i$ acts on $M_i$. Recall that the action of $G$ on $M$ satisfies the \emph{duality condition} of Chen and Eberlein \cite{ChenEberlein} if for every  geodesic $\sigma:\mathbb R\to M$ with endpoints $\sigma^-,\sigma^+\in\partial_\infty M$ there exists a sequence $(g_n)_{n\in\NN}\in G^\NN$ such that for some (hence every) $x\in M$, one has $g_n(x)\to \sigma^+$ and $g_n^{-1}(x)\to \sigma^-$. Since $G$ acts properly and cocompactly on $M$, it follows that the action of $G$ on $M$ satisfies the duality condition, see \cite[p.45]{Bal}.  Therefore the action of  every $G_i$ on $M_i$ also satisfies the duality condition. 

Let $i\leq n$. Assume first that $M_i$ is of rank one. Since the $G_i$-action on $M_i$ satisfies the duality condition, it follows from \cite[Theorem 3.4]{Bal} that there exists a rank one isometry in $G_i$ (acting on $M_i$). Since the action of $G$ on $M$ is nonelementary and has no invariant complete proper submanifold, the action of $G_i$ on $M_i$ is again nonelementary. Hence the action of  $G_i$ on $M_i$  is properly proximal by Theorem~\ref{theo:rank-one}.

If $M_i$ is a higher rank symmetric space of noncompact type, then $H_i=\Isom(M_i)$ is a simple Lie group of higher rank. Let $H_i^0$ be its connected component of the identity, and let $G_i^0=H_i^0\cap G_i$. Then $G_i^0$ is a finite index subgroup of $G_i$. Furthermore the action of $G_i$ on $M_i$ is still nonelementary, and by assumption there is no $G_i$-invariant complete submanifold of $M_i$. Hence $G_i$ is a Zariski-dense subgroup of $H_i$ \cite{Mostow}, and therefore $G_i^0$ is a Zariski-dense subgroup of $G_i$. It then follows from the proof of \cite[Proposition~4.14]{BIP} that the action of $G_i^0$ on $M_i$ is properly proximal. 

If $M_i$ is of rank 1, define $G_i^0=G_i$, and let $G^0=G_1^0\times \dots\times G_n^0$. By Lemma \ref{lemma:product}, the action of $G^0$ on $M$ is properly proximal, and by Lemma~\ref{lemma:proximality} it follows that $G^0$ is properly proximal. Since $G^0$ is of finite index in $G$, we deduce that $G$ is properly proximal by \cite[Proposition~1.6]{BIP}.
\end{proof}

\section{Proper proximality among hierarchically hyperbolic groups}\label{sec:hhs}

The goal of this section is to prove our theorem concerning proper proximality for hierarchically hyperbolic groups. We refer to \cite[Definition~1.1]{BHS} for all relevant definitions in the statement -- see also the brief recap in Section~\ref{sec:def-hhs} below. Every hierarchically hyperbolic space comes with a factor set $\mathfrak{S}$, and (unless $\mathfrak{S}=\emptyset$) with a \emph{principal} hyperbolic space, denoted $\calc S_0$ in the sequel. We say that the $G$-action on $(\calx,\mathfrak{S})$ is \emph{nonelementary} if $\mathfrak{S}\neq\emptyset$ and $G$ acts on $\calc S_0$ with two independent loxodromic isometries (in particular $\calc S_0$ is unbounded). Every hierarchically hyperbolic space also comes with a map $\pi_{S_0}:\calx\to\calc S_0$. We say that the $G$-action on $(\calx,\mathfrak{S})$ \emph{has an almost equivariant principal projection} if there exists $C\ge 0$ such that for every $x\in\calx$ and every $g\in G$, one has $d_{\calc S_0}(\pi_{S_0}(gx),g\pi_{S_0}(x))\le C$.

\begin{theo}\label{theo:hhs}
Let $(\calx,\mathfrak{S})$ be a hierarchically hyperbolic space, with $\calx$ proper.  
Let $G$ be a countable group acting nonelementarily by HHS automorphisms on $(\calx,\mathfrak{S})$, so that the $G$-action on $\calx$ is by uniform quasi-isometries, and has an almost equivariant principal projection. Then the $G$-action on $\calx$ is properly proximal.

If in addition the $G$-action on $\calx$ is proper, then $G$ is properly proximal.
\end{theo}

Notice that as in the $\cat$ setting (see Theorem~\ref{theo:main}), one can also derive a statement for product actions by combining Theorem~\ref{theo:hhs} with Lemmas~\ref{lemma:proximality} and~\ref{lemma:product}.

\begin{rk}
We would like to briefly comment on the seemingly technical assumption in the theorem --  the fact that the $G$-action on $\calx$ is by uniform quasi-isometries and has an almost equivariant principal projection. This does not formally follow from the definition of an action by HHS automorphisms, but are natural to add in many situations, see e.g.\ the paragraph after \cite[Definition~1.11]{DHS}. These assumptions are quite harmless in the sense that they hold in all interesting situations we have in mind: typically, in the case of the mapping class group of a connected, orientable surface $\Sigma$ of finite type, the collection of all isotopy classes of essential subsurfaces of $\Sigma$ is countable, and the action of $\Mod(\Sigma)$ on the marking graph is by isometries, with an almost equivariant projection map to the curve graph.  
\end{rk}

\subsection{Review on hierarchically hyperbolic spaces and their boundaries}\label{sec:def-hhs}

We refer the reader to \cite[Definition~1.1]{BHS} for the full definition of a hierarchically hyperbolic space, and simply recall the setting and the facts that will be relevant for us, with an emphasis on the mapping class group case -- see \cite[Theorem~11.1]{BHS} for why mapping class groups are hierarchically hyperbolic.

\paragraph*{Hierarchically hyperbolic spaces.} A hierarchically hyperbolic space comes as a pair $(\calx,\mathfrak{S})$, where $\calx$ is a quasigeodesic metric space (taken to be the marking graph in the mapping class group setting), and $\mathfrak{S}$ is an index set (the set of all isotopy classes of essential subsurfaces). From now on we will assume that $\mathfrak{S}\neq\emptyset$. The set $\mathfrak{S}$ has a distinguished element $S_0$, namely, the unique $\sqsubseteq$-maximal element in the notation from \cite{BHS} (for surfaces, $S_0=\Sigma$ is the whole surface). For every $S\in\mathfrak{S}$, we have a $\delta$-hyperbolic space $\calc S$, for some uniform $\delta\ge 0$ (the curve graph of the subsurface $S$, whose hyperbolicity was proved by Masur and Minsky in \cite{MM1}). 

For every $S\in\mathfrak{S}$, there is a coarsely Lipschitz (with uniform constants) map $\pi_S:\calx\to 2^{\calc S}$ with nonempty values, and there exists $K\ge 0$ such that for every $S\in\mathfrak{S}$ and every $x\in \calx$, one has $\diam_{\calc S}(\pi_S(x))\le K$. In the surface setting, the maps $\pi_S$ correspond to subsurface projections as introduced by Masur and Minsky in \cite{MM}. In addition, for every $S\neq S_0$, there is a projection map $\rho_S^{S_0}:\calc S_0\to 2^{\calc S}$ with nonempty values. Again, in the mapping class group setting, this is defined from subsurface projection: a possibility is to define $\rho_S^{S_0}(c)$ as the set of all isotopy classes of curves $c'$ on $S$ that are disjoint from some connected component of $S\cap c$; this has bounded diameter in $\calc S$ whenever $S\cap c$ is nonempty -- if this intersection is empty on simply lets $\rho_S^{S_0}(c)=V(\calc S)$, which then has unbounded diameter. Conversely, there is a bounded region $X_S\subseteq \calc S_0$ associated to every $S\in\mathfrak{S}$ (for surfaces, one can take $X_S$ to be the finite set of all isotopy classes of boundary curves of $S$). More generally, there are projection maps between various sets $\calc S$, but they will not play any role in the present paper. These projections satisfy several consistency relations. The following version of one of these relations will be important for us: there exists a constant $\kappa_0\ge 0$ such that for every $x\in\calx$ and every $S\in\mathfrak{S}$, either $d_{\calc S_0}(\pi_{S_0}(x),X_{S})\le\kappa_0$ or $\diam_{\calc S}(\pi_S(x)\cup\rho_S^{S_0}(\pi_{S_0}(x)))\le\kappa_0$.

There is a relation of \emph{orthogonality} between elements of $\mathfrak{S}$: for surfaces, two subsurfaces are orthogonal whenever they have disjoint representatives in their respective isotopy classes. There is a bound on the cardinality of a set of pairwise orthogonal elements of $\mathfrak{S}$.

We will also need the following fact, which follows from the axiom of partial realization from \cite[Definition~1.1]{BHS}: there exists $\alpha\ge 0$ such that for every $p\in\calc S_0$, there exists $\tilde{p}\in\calx$ such that $d_{\calc S_0}(\pi_{S_0}(\tilde{p}),p)\le\alpha$.

Finally, we will need the following simplified version of the bounded geodesic image axiom -- only stated here for projections from $\calc S_0$ to the various $\calc S$: there exists $E\ge 0$ such that for every $S\in\mathfrak{S}$, and every geodesic $\gamma$ in $\calc S_0$, either $\diam_{\calc S}(\rho^{S_0}_S(\gamma))\le E$, or else $\gamma$ meets the $E$-neighborhood of $X_S$ in $\calc S_0$. In the mapping class group setting, the bounded geodesic image theorem was proved by Masur and Minsky in \cite[Theorem~3.1]{MM}.

\paragraph*{Automorphisms and nonelementarity.} We refer to \cite[Definition~1.21]{BHS} for the notion of an \emph{automorphism} of a hierarchically hyperbolic space, referred to as a \emph{HHS automorphism} in the present paper (roughly, this is a map that preserves the hierarchical structure and has a quasi-inverse that also does so). Every automorphism of a hierarchically hyperbolic space comes with a permutation $\phi$ of the index set $\mathfrak{S}$ and a family of isometries from $\calc S$ to $\calc \phi(S)$, one for each $S\in\mathfrak{S}$. Given a group $G$ acting on $(\calx,\mathfrak{S})$ by automorphisms, we say that the $G$-action on $(\calx,\mathfrak{S})$ is \emph{nonelementary} if $G$ contains two elements that act as independent loxodromic isometries on $\calc S_0$ (i.e.\ with disjoint fixed sets in the boundary). 

\paragraph*{Boundary of a hierarchically hyperbolic space.} The \emph{HHS boundary} of a hierarchically hyperbolic quasigeodesic metric space $\calx$ was defined by Durham, Hagen and Sisto in \cite{DHS}. Every point $\xi\in\partial\calx$ is represented as a sum $\sum_S a_S^\xi \xi_S$, taken over a nonempty finite set of pairwise orthogonal $S\in\mathfrak{S}$, where each $\xi_S$ is a point in $\partial_\infty \calc S$, and each $a_S^\xi$ is a positive real number, with $\sum_S a_S^\xi=1$. The \emph{support} of $\xi$, denoted by $\Supp(\xi)$, is defined as the finite set of all $S\in\mathfrak{S}$ that arise in the sum which defines $\xi$. By convention, when $S$ does not belong to the support of $\xi$, we let $a_S^\xi=0$. 

We start by giving a criterion for when a sequence $(x_n)_{n\in\mathbb{N}}\in\calx^\mathbb{N}$ converges to a boundary point, which is a consequence of \cite[Definition~2.10]{DHS}.

\begin{fact}\label{fact:convergence-criterion-interior}
Let $(x_n)_{n\in\mathbb{N}}\in\calx^\mathbb{N}$, let $\xi=\sum_{S}a_S^\xi \xi_S\in\partial\calx$. Then $(x_n)_{n\in\mathbb{N}}$ converges to $\xi$ if and only if the following two conditions hold:
\begin{enumerate}
\item for every $S\in\Supp(\xi)$, the sequence $(\pi_S(x_n))_{n\in\mathbb{N}}$ converges to $\xi_S$ in $\calc S\cup\partial_\infty\calc S$,
\item for every $S\in\mathfrak{S}$ which is either equal to an element of $\Supp(\xi)$, or orthogonal to every element of $\Supp(\xi)$, for every $S'\in\Supp(\xi)$, one has 
$$\frac{d_{\calc S}(\pi_S(x_0),\pi_S(x_n))}{d_{\calc S'}(\pi_{S'}(x_0),\pi_{S'}(x_n))}\to \frac{a_S^\xi}{a_{S'}^\xi}$$
as $n$ goes to $+\infty$. 
\end{enumerate}  
\end{fact}

As the projection maps $\pi_S$ are coarsely Lipschitz, the limit $\xi$ does not change if $(x_n)_{n\in\mathbb{N}}$ is replaced by a sequence $(x'_n)_{n\in\mathbb{N}}$ for which $d_\calx(x_n,x'_n)$ is uniformly bounded.

For every $S\in\mathfrak{S}$, the Gromov boundary $\partial_\infty\calc S$ topologically embeds in $\partial\calx$ by \cite[Proposition~2.13]{DHS}. In particular, we will view $\partial_\infty \calc S_0$ as a subspace of $\partial\calx$. The following fact follows from Fact~\ref{fact:convergence-criterion-interior}.

\begin{fact}\label{fact:convergence-main}
Let $\xi\in\partial_\infty\calc S_0$, and let $(x_n)_{n\in\mathbb{N}}\in\calx^\mathbb{N}$ be a sequence such that $(\pi_{S_0}(x_n))_{n\in\mathbb{N}}$ converges to $\xi$ for the topology on $\calc S_0\cup\partial_\infty\calc S_0$. Then $(x_n)_{n\in\mathbb{N}}$ converges to $\xi$ for the topology on $\calx\cup\partial\calx$. 
\end{fact}

We will also need to understand when a sequence $(\xi_n)_{n\in\mathbb{N}}\in (\partial_\infty \calc S_0)^\mathbb{N}$ converges to a point $\xi\in\partial\calx\setminus\partial_\infty\calc S_0$. For that, we need the following definition. Given $S\in\mathfrak{S}$, there is a (coarsely well-defined) notion of a \emph{boundary projection} $\partial\pi_S:\partial_\infty\calc S_0\to\calc S$, as follows. Recall that associated to $S$ is a bounded subset $X_S$ of $\calc S_0$. Now, given $\xi\in\partial_\infty\calc S_0$, take a $(1,20\delta)$-quasigeodesic ray $\gamma_\xi$ in $\calc S_0$ joining $X_S$ to $\xi$ (see e.g.\ \cite[Remark~2.16]{KB} for the existence of such a ray). By the bounded geodesic image axiom, the ray $\gamma_\xi$ has an infinite subray on which the projection $\rho^{S_0}_S$ is coarsely constant, and one defines $\partial\pi_S(\xi)$ as a point in the $\rho^{S_0}_S$-image of this subray. The following convergence criterion directly follows from the definition of the topology on $\calx\cup\partial\calx$, see \cite[Definition~2.8]{DHS}.

\begin{fact}\label{fact:convergence-criterion}
Let $(\xi_n)_{n\in\mathbb{N}}\in (\partial_\infty\calc S_0)^{\mathbb{N}}$, let $\xi=\sum_{S}a_S^\xi \xi_S\in\partial\calx\setminus\partial_\infty\calc S_0$. Then $(\xi_n)_{n\in\mathbb{N}}$ converges to $\xi$ if and only if the following two conditions hold:
\begin{enumerate}
\item for every $S\in\Supp(\xi)$, the sequence $(\partial\pi_S(\xi_n))_{n\in\mathbb{N}}$ converges to $\xi_S$ in $\calc S\cup\partial_\infty\calc S$,
\item for every $S\in\mathfrak{S}$ which is equal to some element of $\Supp(\xi)$, or orthogonal to every element of $\Supp(\xi)$, for every $S'\in\Supp(\xi)$, one has 
$$\frac{d_{\calc S}(\pi_S(x_0),\partial\pi_S(\xi_n))}{d_{\calc S'}(\pi_{S'}(x_0),\partial\pi_{S'}(\xi_n))}\to\frac{a_S^\xi}{a_{S'}^\xi}$$ as $n$ goes to $+\infty$. 
\end{enumerate}  
\end{fact}

 We will also need the following topological fact.

\begin{lemma}\label{lemma:boundary-hhs-topology}
Let $(\calx,\mathfrak{S})$ be a hierarchically hyperbolic space, with $\calx$ proper and $\mathfrak{S}$ countable. Then $\calx\cup\partial\calx$ is compact and metrizable.

\end{lemma}

\begin{proof}
The compactness of $\calx\cup\partial\calx$ was proved in \cite[Theorem~3.4]{DHS}. The metrizability of $\calx\cup\partial\calx$ was proven by Hagen in \cite[Proposition~2.1]{Hag}.
\end{proof}

\subsection{Dynamics on the boundary and proper proximality}

In order to prove Theorem~\ref{theo:hhs}, we will apply Lemma~\ref{lemma:north-south} with $K=\partial\calx$ and $K^*=\partial_\infty\calc S_0$, which contains a Cantor set because the $G$-action on $\calc S_0$ is assumed to be nonelementary. The key point is to prove the following statement, which can be viewed as a version for hierarchically hyperbolic spaces of the theorem of Papasoglu and Swenson we used in Section~\ref{sec:cat}.

\begin{prop}\label{prop:dynamics-boundary-hhs}
Let $(\calx,\mathfrak{S})$ be a hierarchically hyperbolic space, with $\calx$ proper. 
 Let $G$ be a countable group acting nonelementarily by HHS automorphisms on $(\calx,\mathfrak{S})$ with an almost equivariant principal projection, so that the $G$-action on $\calx$ is by uniform quasi-isometries. Let $x\in\calx$, and let $(g_n)_{n\in\mathbb{N}}\in G^{\mathbb{N}}$ be a sequence that escapes every compact subspace of $\calx$. 

Then there exist $\xi^-,\xi^+\in\partial\calx$ and a subsequence $(g_{\sigma(n)})_{n\in\mathbb{N}}$ such that for every $\xi\in\partial_\infty \calc S_0\setminus\{\xi^-\}$, one has $\lim\limits_{n\to +\infty} g_{\sigma(n)}\xi=\xi^+$.
\end{prop}

For the proof, we recall that given three points $x,y,z$ in a hyperbolic space $X$, the \emph{Gromov product} $\langle x,y\rangle_z$ is defined as \[\langle x,y\rangle_z=\frac{1}{2}\left(d(z,x)+d(z,y)-d(x,y)\right).\] When $X$ is geodesic, it measures the distance of $z$ to the geodesic segment $[x,y]$, up to a bounded error that only depends on the hyperbolicity constant of $X$. Given $\xi\in\partial_\infty X$, one also defines \[\langle x,\xi\rangle_z=\sup (\liminf \langle x,y_n\rangle_z),\] where the supremum is taken over all sequences $(y_n)_{n\in\mathbb{N}}\in X^{\mathbb{N}}$ that converge to $\xi$. 

\begin{proof}
Let $x\in\calx$. As $\calx$ is proper and $(g_n)_{n\in\mathbb{N}}$ escapes every compact subspace of $\calx$, the distance $d(x,g_nx)$ diverges to $+\infty$. As the $G$-action on $\calx$ is by uniform quasi-isometries, it follows that $d(x,g_n^{-1}x)$ also diverges to $+\infty$. This implies that $(g_n^{-1}x)_{n\in\mathbb{N}}$ leaves every compact subset of $\calx$. Since $\calx\cup\partial\calx$ compactifies $\calx$, up to passing to a subsequence, we can therefore assume that $(g_n x)_{n\in\mathbb{N}}$ converges to a point $\xi^+\in\partial\calx$, and $(g_n^{-1} x)_{n\in\mathbb{N}}$ converges to a point $\xi^-\in\partial\calx$. As follows from Fact~\ref{fact:convergence-criterion-interior} and the fact that the $G$-action on $\calx$ is by uniform quasi-isometries, these limits are unchanged if $x$ is replaced by any other point $y\in\calx$. 

We claim that for every $p\in\calc S_0$ and every $\xi\in\partial_\infty\calc S_0\setminus\{\xi^-\}$, the sequence $(g_n^{-1}p)_{n\in\mathbb{N}}$ does not converge to $\xi$ in $\calc S_0\cup\partial_\infty\calc S_0$. Indeed, as the $G$-action on $\calc S_0$ is by isometries, it is enough to prove the claim for some point $p\in\calc S_0$; we choose $p=\pi_{S_0}(x)$. As $(g_n^{-1} x)_{n\in\mathbb{N}}$ does not converge to $\xi$ in $\calx\cup\partial\calx$, it follows from Fact~\ref{fact:convergence-main} that $(\pi_{S_0}(g_n^{-1}x))_{n\in\mathbb{N}}$ does not converge to $\xi$ in $\calc S_0\cup\partial_\infty\calc S_0$. As the $G$-action on $\calx$ has an almost equivariant principal projection (i.e.\ $\pi_{S_0}$ is equivariant up to a bounded error), we deduce that $(g_n^{-1}\pi_{S_0}(x))_{n\in\mathbb{N}}$ does not converge to $\xi$ in $\calc S_0\cup\partial_\infty\calc S_0$. This proves our claim.  

Let $\xi\in\partial_\infty\calc S_0\setminus\{\xi^-\}$. Since $(g_n^{-1}p)_{n\in\mathbb{N}}$ does not converge to $\xi$ in $\calc S_0\cup\partial_\infty\calc S_0$, using hyperbolicity of $\calc S_0$, there exists a constant $C_1\ge 0$ and an open neighborhood $U$ of $\xi$ in $\calc S_0\cup\partial_\infty\calc S_0$ such that for every $n\in\mathbb{N}$ and every $z\in U$ lying on a $(1,20\delta)$-quasigeodesic from $p$ to $\xi$, the  Gromov product $\langle g_n^{-1}p,\xi\rangle_z$ is at most $C_1$. As the $G$-action on $\calc S_0$ is by isometries, this implies that $\langle p,g_n\xi\rangle_{g_n z}\le C_1$. 

We aim to show that $(g_n\xi)_{n\in\mathbb{N}}$ converges to $\xi^+$, and for that there are now two cases to consider. 
\\
\\
\textbf{Case 1:} We have $\xi^+\in\partial_\infty\calc S_0$.
\\

In this case, we first claim that for every $p\in\calc S_0$, the sequence $(g_np)_{n\in\mathbb{N}}$ converges to $\xi^+$ in $\calc S_0\cup\partial_\infty\calc S_0$. As the $G$-action on $\calc S_0$ is by isometries, it is enough to prove this claim when $p=\pi_{S_0}(x)$. As $(g_nx)_{n\in\mathbb{N}}$ converges to $\xi^+$ in $\calx\cup\partial\calx$, it follows from Fact~\ref{fact:convergence-criterion-interior} that $(\pi_{S_0}(g_nx))_{n\in\mathbb{N}}$ converges to $\xi^+$ in $\calc S_0\cup\partial_\infty\calc S_0$. As the $G$-action on $\calx$ has an almost equivariant principal projection, the claim follows.   

Choosing a point $z\in U$ as in the paragraph preceding Case~1, the sequence $(g_nz)_{n\in\mathbb{N}}$ converges to $\xi^+$ in $\calc S_0\cup\partial_\infty\calc S_0$. The inequality $\langle p,g_n\xi\rangle_{g_n z}\le C_1$ then implies that $(g_n\xi)_{n\in\mathbb{N}}$ converges to $\xi^+$ in $\partial_\infty\calc S_0$. As the inclusion map $\partial_\infty\calc S_0\to\partial\calx$ is a topological embedding, this convergence also holds in $\partial\calx$, as desired.
\\
\\
\textbf{Case 2:} We have $\xi^+\notin\partial_\infty\calc S_0$.
\\

By definition of $\partial\calx$, the point $\xi^+$ is written as a finite sum $\sum_S a_S\xi^+_S$ over a collection of pairwise orthogonal $S\in\mathfrak{S}$, where each $\xi^+_S$ belongs to $\partial_\infty\calc S$, each $a_S$ is a positive real number, with $\sum_S a_S=1$. 

The fact that $(g_n\xi)_{n\in\mathbb{N}}$ converges to $\xi^+$ follows from the convergence criteria recalled in Fact~\ref{fact:convergence-criterion-interior} and Fact~\ref{fact:convergence-criterion} together with the following claim: there exists $\tilde y\in\calx$ such that for every $S\in\mathfrak{S}$ which either belongs to $\Supp(\xi^+)$ or is orthogonal to every element of $\Supp(\xi^+)$, the projections $\partial\pi_{S}(g_n\xi)$ lie at bounded distance from $\pi_S(g_n\tilde{y})$. 

We now prove the above claim. We start by recalling a few notations. First, recall that the $G$-action on $\calx$ has an almost equivariant principal projection: this means in particular that there exists a constant $C\ge 0$ such that for every $x\in\calx$ and every $n\in\mathbb{N}$, one has $\diam_{\calc S_0}(\pi_{S_0}(g_nx)\cup g_n\pi_{S_0}(x))\le C$. Also, let $\alpha\ge 0$ be a constant such that for every $p\in\calc S_0$, there exists $\tilde p\in\calx$ with $d_{\calc S_0}(\pi_{S_0}(\tilde p),p)\le\alpha$. Let $K$ be the maximal diameter in $\calc S_0$ of a set of the form $\pi_{S_0}(\tilde{x})$ with $\tilde{x}\in\calx$. Recall that $X_S$ denotes the bounded set associated to $S$ in $\calc S_0$. Choose a point $p\in X_{S}$. By \cite[Lemma~1.5]{DHS}, the diameter of the union of all subspaces $X_S$ with $S$ either in $\Supp(\xi^+)$ or orthogonal to every element of $\Supp(\xi^+)$ is finite; we denote it by $L$. Let $E\ge 0$ be the constant coming from the bounded geodesic image axiom, and let $\kappa_0\ge 0$ be the constant coming from the axiom of consistency of projections as recalled in the previous section. 

Let $C_1$ be as in the paragraph preceding Case~1. Let $\beta\ge 0$ be such that for every $x,y\in\calc S_0$ and every $\eta\in\partial_\infty\calc S_0$, if $\langle x,\eta\rangle_{y}\le C_1+\alpha+C$ and $d_{\calc S_0}(x,y)\ge \beta$, then every $(1,20\delta)$-quasigeodesic from $y$ to $\eta$ lies outside the $(E+L+\kappa_0+1)$-neighborhood of $x$. Let $C_2:=\alpha+K+\beta+C$.

As $(g_n^{-1}p)_{n\in\mathbb{N}}$ does not converge to $\xi$ in $\calc S_0\cup\partial_\infty\calc S_0$, we can find $y\in\calc S_0$ such that for every $n\in\mathbb{N}$, one has $\langle g_n^{-1}p,\xi\rangle_y\le C_1$ and $d_{\calc S_0}(g_n^{-1}p,y)\ge C_2$ (such a point $y$ can be found in a quasigeodesic from $p$ to $\xi$, sufficiently far from $p$, see the discussion before Case 1). Let $\tilde y\in\calx$ be such that $d_{\calc S_0}(y,\pi_{S_0}(\tilde y))\le\alpha$. In particular, we get that $d_{\calc S_0}(g_n^{-1}p,\pi_{S_0}(\tilde y))\ge C_2-(\alpha+K)$.

Applying the isometry $g_n$, we deduce that for every $n\in\mathbb{N}$, one has $\langle p,g_n\xi\rangle_{g_n y}\le C_1$ and $d_{\calc S_0}(p,\pi_{S_0}(g_n\tilde{y}))\ge C_2-\alpha-K-C=\beta$. Therefore, for every $y'\in\pi_{S_0}(g_n\tilde{y})$ and every $n\in\mathbb{N}$, we have $d(g_ny,y')\le\alpha+C$ whence $\langle p,g_n\xi\rangle_{y'}\le C_1+\alpha+C$ and $d_{\calc S_0}(p,y')\ge\beta$, so every $(1,20\delta)$-quasigeodesic ray from $y'$ to $g_n\xi$ lies outside the $(E+L+\kappa_0+1)$-neighborhood of $p$. For every $S\in\mathfrak{S}$ which either belongs to $\Supp(\xi^+)$ or is orthogonal to every element in $\Supp(\xi^+)$, such a quasigeodesic ray thus lies outside the $(E+\kappa_0+1)$-neighborhood of $X_S$. 

Let $S\in\mathfrak{S}$ which either belongs to $\Supp(\xi^+)$ or is orthogonal to every element in $\Supp(\xi^+)$. Using the bounded geodesic image property, we deduce that the diameter $$\diam_{\calc S}(\{\partial\pi_S(g_n\xi)\}\cup  \rho^{S_0}_S(\pi_{S_0}(g_n\tilde{y})))$$ is uniformly bounded. Since $\pi_{S_0}(g_n\tilde{y})$ lies outside the $(\kappa_0+1)$-neighborhood of $X_S$, it follows from the consistency axiom that $\diam_{\calc S}(\pi_S(g_n\tilde{y})\cup\rho^{S_0}_S(\pi_{S_0}(g_n\tilde{y})))\le\kappa_0$.  This proves our claim, and concludes the proof. 
\end{proof}

We are now in position to complete our proof of the main theorem of this section.

\begin{proof}[Proof of Theorem~\ref{theo:hhs}]
We will use the dynamical criterion provided by Lemma~\ref{lemma:north-south}. Let $K:=\calx\cup\partial\calx$, which is compact and metrizable (Lemma~\ref{lemma:boundary-hhs-topology}). Let $\calc$ be the principal hyperbolic space in the hierarchically hyperbolic structure of $\calx$, and let $K^*:=\partial_\infty\calc$: this contains a topologically embedded Cantor set because the $G$-action on $\calc$ is assumed to be nonelementary, so $\partial_\infty\calc$ contains the limit set of any Schottky subgroup of $G$. The nonelementarity of the action also ensures that $K$ does not carry any $G$-invariant probability measure. Finally, the main dynamical assumption required by Lemma~\ref{lemma:north-south} was checked in Proposition~\ref{prop:dynamics-boundary-hhs}. This completes our proof. 
\end{proof}

\subsection{Applications to mapping class groups and their subgroups}

A \emph{surface of finite type} is a topological space $\Sigma$ obtained from a (possibly disconnected) boundaryless compact surface by possibly removing finitely many points and finitely many open disks. Its \emph{extended mapping class group} $\Mod^*(\Sigma)$ is the group of all homeomorphisms of $\Sigma$ that fix its boundary pointwise, up to homotopies that are the identity on every boundary component. As a consequence of Theorem~\ref{theo:hhs}, we deduce the following statement, thus answering a question raised by Boutonnet, Ioana and Peterson in \cite{BIP}.

\begin{cor}
Let $\Sigma=\Sigma_{g,n}$ be a connected orientable surface obtained from a closed surface of genus $g$ by removing $n$ points, with $3g+n-3>0$. Then $\Mod^*(\Sigma)$ is properly proximal.
 
More generally, every subgroup of $\Mod^*(\Sigma)$ acting nonelementarily on the curve graph of $\Sigma$ is properly proximal.
\end{cor}

Notice that the extended mapping class group of a closed torus ($g=1$ and $n=0$) is isomorphic to $\mathrm{GL}(2,\mathbb{Z})$ and is therefore properly proximal. On the other hand, the mapping class group of a sphere with at most $3$ punctures ($g=0$ and $n\le 3$) is finite.

\begin{proof}
By  \cite[Theorem~11.1]{BHS}, the group $\Mod^*(\Sigma)$ is hierarchically hyperbolic, and in fact satisfies the assumptions from Theorem~\ref{theo:hhs}. The principal hyperbolic space in the HHS structure is the curve graph of $\Sigma$. The conclusion thus follows from Theorem~\ref{theo:hhs}. 
\end{proof}

More generally, we classify subgroups of $\Mod^*(\Sigma)$ that are properly proximal. We recall that the \emph{FC-center} $\WZ(G)$ of a group $G$ is defined as the subgroup of $G$ made of all elements that centralize a finite-index subgroup of $G$. 

\begin{theo}\label{theo:mcg}
Let $\Sigma$ be a connected orientable surface of finite type, and let $H$ be an infinite subgroup of $\Mod^*(\Sigma)$.

Then $H$ is properly proximal if and only if $\WZ(H)$ is finite.
\end{theo}

\begin{rk}
Notice in particular that the mapping class group of a surface of finite type with boundary is never properly proximal as the group of peripheral twists is central.

We mention that the connectedness assumption in Theorem~\ref{theo:mcg} can in fact be removed by working in the finite-index subgroup of $\Mod^*(\Sigma)$ that preserves each of the finitely many connected components. We have decided to write the proof in the case of a connected surface to improve its readability.

Also, the case where $\Sigma$ is a closed non-orientable surface can be treated by observing that, unless $\Sigma$ is a projective plane or a Klein bottle, $\Mod^*(\Sigma)$ virtually embeds in the mapping class group of an orientable double cover of $\Sigma$, see \cite[Lemma~4]{Fuj}.
\end{rk}   

\begin{proof}
If $WZ(H)$ is infinite, then $H$ is inner amenable in the sense that there exists an atomless mean on $H$ which is invariant under the $H$-action on itself by conjugation. Therefore $H$ cannot be properly proximal \cite[Proposition~1.6]{BIP}. 

Conversely, let $H$ be an infinite subgroup of $\Mod^*(\Sigma)$ such that $\WZ(H)$ is finite; we aim to show that $H$ is properly proximal.

If $\Sigma$ has nonempty boundary, let $\widehat{\Sigma}$ be a surface obtained from $\Sigma$ by gluing a once-holed disk on every boundary component of $\Sigma$. Then the inclusion $\Sigma\hookrightarrow\widehat{\Sigma}$ yields a homomorphism $\Mod^*(\Sigma)\to\Mod^*(\widehat{\Sigma})$ whose kernel is central and free abelian (made of boundary twists), see \cite[Proposition~3.19]{FM-primer}. As $\WZ(H)$ is finite, this homomorphism is injective when restricted to $H$. This enables us to reduce to the case where $\Sigma$ is boundaryless, so from now on we will work under this assumption. 

We can assume that $\Sigma$ is not a sphere with at most $3$ punctures (otherwise $\Mod^*(\Sigma)$ is finite) and is not a closed torus (otherwise $\Mod^*(\Sigma)$ is virtually free and the conclusion holds true). Let $H^0\subseteq H$ be a finite-index subgroup of $H$ made of mapping classes that act trivially on homology modulo $3$. A useful property of this subgroup is the following \cite[Corollary~3.6]{Iva}: every isotopy class of essential simple closed curves on $\Sigma$ which is virtually $H^0$-invariant, is in fact $H^0$-invariant. As proper proximality is a commensurability invariant \cite[Proposition~1.6]{BIP}, it is enough to prove that $H^0$ is properly proximal.

Let $X$ be a maximal collection of pairwise disjoint and pairwise non-isotopic essential simple closed curves on $\Sigma$ whose isotopy classes are $H^0$-invariant. Let $\mathbb{Y}$ be the finite set made of all connected components of $\Sigma\setminus X$. Then there is a homomorphism $$H^0\to\prod_{S\in\mathbb{Y}}\Mod^*(S),$$ whose kernel is central and abelian (made of multitwists about the curves in $X$). As $\WZ(H)$ is trivial, the kernel of this homomorphism is trivial, so $H^0$ embeds in a direct product of mapping class groups. As $H$ is infinite and $\WZ(H)$ is finite, up to replacing $\mathbb{Y}$ by a nonempty subset $\mathbb{Y}'$, we can assume that the image of $H^0$ in each factor is not virtually abelian.

We claim that for every subsurface $S\in\mathbb{Y}'$, the image $H_S$ of $H^0$ in $\Mod^*(S)$ contains two independent pseudo-Anosov mapping classes, and therefore acts nonelementarily on the curve graph of $S$. Indeed, if $H_S$ does not contain any pseudo-Anosov mapping class of $S$, then it follows from Ivanov's subgroup classification theorem \cite[Theorem~1]{Iva} that $H^0$ virtually fixes the isotopy class of an $S$-essential simple closed curve $c$ on $S$. As $H^0$ acts trivially on homology modulo $3$, the isotopy class of $c$ is in fact $H^0$-invariant \cite[Corollary~3.6]{Iva}, contradicting the maximality of $X$. Now, if $H_S$ contains a pseudo-Anosov mapping class of $S$ but does not contain two independent pseudo-Anosov mapping classes, then $H_S$ is virtually abelian \cite[Theorem~2]{Iva}, a contradiction. This completes the proof of our claim. 

Proper proximality of $H^0$ follows from the above claim in view of Lemma~\ref{lemma:product} and Theorem~\ref{theo:hhs}, using the hierarchically hyperbolic structure on mapping class groups given in \cite[Theorem~11.1]{BHS}, for which the principal hyperbolic space is the curve graph of the surface.   
\end{proof}

\section{Applications to von Neumann algebras}\label{sec:von-Neumann}

As mentioned in the introduction, our main motivation to prove proper proximality is to deduce rigidity properties for associated von Neumann algebras by applying the main result from \cite{BIP}.

\subsection{Von Neumann algebras associated to properly proximal groups}

Let $G\actson (X,\mu)$ be an ergodic measure-preserving essentially free action of a countable group $G$ on a standard probability space $(X,\mu)$. The rigidity results in \cite{BIP} involve considering group actions as above that are \emph{weakly compact} in the sense of Ozawa and Popa; we refer to \cite[Definition~3.1]{OzawaPopa1} for the definition. This notion generalizes the notion of the action being \emph{compact}, i.e.\ such that the image of $G$ in $\Aut(X,\mu)$ is relatively compact. In particular, when $G$ is residually finite, the $G$-action on its profinite completion (equipped with the Haar measure) is compact. We refer to \cite[Section~3]{OzawaPopa1} or \cite[Section~2.4]{Ioa} for more details. We mention that the groups considered in the present paper include many classes of groups that are residually finite. In particular, mapping class groups of finite-type surfaces are residually finite \cite{Gro}.

We recall that a \emph{Cartan subalgebra} of a von Neumann algebra $M$ is a maximal abelian subalgebra of $M$ whose normalizer in $M$ generates $M$. We refer to \cite[Definition~2.2]{BIP} for the definition of a \emph{weakly compact} Cartan subalgebra, and \cite[Proposition~3.2]{OzawaPopa} for the relationship to a weakly compact group action. The following theorem from \cite{BIP} applies to all groups considered in the present paper.

\begin{theo}[Boutonnet--Ioana--Peterson \cite{BIP}]\label{cor:von-neumann}
Let $G$ be a properly proximal countable group. Then the following hold.
\begin{enumerate}
\item The group von Neumann algebra $LG$ has no weakly compact Cartan subalgebra.
\item If $G\actson (X,\mu)$ is an ergodic measure-preserving essentially free action of $G$ on a standard probability space $(X,\mu)$, then
\begin{enumerate}
\item if the action $G\actson (X,\mu)$ is weakly compact, then $L^\infty(X)\rtimes G$ admits $L^\infty(X,\mu)$ as its unique weakly compact Cartan subalgebra, up to unitary conjugacy;
\item if the action $G\actson (X,\mu)$ is not weakly compact, then $L^\infty(X)\rtimes G$ does not contain any weakly compact Cartan subalgebra. 
\end{enumerate}
\end{enumerate}
\end{theo}

This has applications to rigidity questions. Namely, let $G$ be one of the groups considered in the present paper, and let $G\actson (X,\mu)$ be an ergodic measure-preserving essentially free action of $G$ on a standard probability space $(X,\mu)$. Let $H\curvearrowright (Y,\nu)$ be a weakly compact ergodic measure-preserving essentially free action of a countable group $H$ on a standard probability space $(Y,\nu)$. If $L^\infty(X,\mu)\rtimes G$ is isomorphic to $L^\infty(Y,\nu)\rtimes H$, then the action $G\actson (X,\mu)$ is weakly compact, and any isomorphism of von Neumann algebra between $L^\infty(X,\mu)\rtimes G$ and $L^\infty(Y,\nu)\rtimes H$ sends $L^\infty(X,\mu)$ to $L^\infty(Y,\nu)$ up to unitary conjugacy, so it follows from a theorem of Singer \cite{Sin} that the actions $G\curvearrowright (X,\mu)$ and $H\curvearrowright (Y,\nu)$ are orbit equivalent. In the presence of orbit equivalence rigidity results -- e.g.\ for mapping class groups by work of Kida \cite{Kid,Kid2}, see Section~\ref{sec:von-Neumann-mcg} below -- this leads to even stronger rigidity results.

\subsection{$\calc$-rigidity of $\cat$ cubical groups}

In the case of groups acting on $\cat$ cube complexes, we get a stronger statement which does not require restricting to weakly compact actions. This follows from works of Ozawa and Popa \cite[Theorem~2.3]{OzawaPopa} and Popa and Vaes \cite[Theorem~1.2 and~Remark~1.3]{PV} under the assumption that hyperplane stabilizers are not co-amenable (see also \cite[Corollary~A]{OzawaPopa} in the more restrictive case of profinite actions). The second property of the following statement is often called \emph{$\calc$-rigidity} of $G$. 

\begin{cor}\label{cor:neumann-cat}
Let $G$ be a countable group acting properly nonelementarily by cubical automorphisms on a proper finite-dimensional $\cat$ cube complex. Then the following hold.
\begin{enumerate}
\item The group von Neumann algebra $LG$ does not contain any Cartan subalgebra.
\item If $G\actson (X,\mu)$ is an ergodic measure-preserving essentially free action of $G$ on a standard probability space $(X,\mu)$, then $L^\infty(X)\rtimes G$ contains $L^\infty(X)$ as its unique Cartan subalgebra, up to unitary conjugacy.
\end{enumerate}
\end{cor}

\begin{proof}
The group $G$ is properly proximal by Theorem~\ref{theo:ccc}, and weakly amenable by a theorem of Guentner and Higson \cite{GH}. The conclusion thus follows from \cite[Theorem~1.5]{BIP}. 
\end{proof}

Again, this has applications to rigidity questions. Namely, let $G$ be a group acting properly nonelementarily by cubical automorphisms on a proper finite-dimensional $\cat$ cube complex, and let $G\actson (X,\mu)$ be an ergodic measure-preserving essentially free $G$-action on a standard probability space $(X,\mu)$. Let $H\curvearrowright (Y,\nu)$ be an ergodic measure-preserving essentially free action of a countable group $H$ on a standard probability space $(Y,\nu)$. If $L^\infty(X,\mu)\rtimes G$ is isomorphic to $L^\infty(Y,\nu)\rtimes H$, then the actions $G\curvearrowright (X,\mu)$ and $H\curvearrowright (Y,\nu)$ are orbit equivalent. 

\subsection{Superrigidity of weakly compact actions of mapping class groups}\label{sec:von-Neumann-mcg}

In the case of mapping class groups, we can combine Theorem~\ref{cor:von-neumann} with work of Kida \cite{Kid2,Kid}  establishing the orbit equivalence rigidity of mapping class group actions, to get an even stronger rigidity statement.

Recall that two actions $G\actson (X,\mu)$ and $H\actson (Y,\nu)$ are \emph{stably conjugate} if there exist finite-index subgroups $G^0\subseteq G$ and $H^0\subseteq H$, finite normal subgroups $F_G\unlhd G^0$ and $F_H\unlhd H^0$ and actions $G^0\actson (X_0,\mu_0)$ and $H^0\actson (Y_0,\nu_0)$ such that the original actions of $G$ and $H$ are induced from the actions of $G^0$ and $H^0$, and the actions $G^0/F_G\actson X_0/F_G$ and $H^0/F_H\actson X_0/F_H$ are conjugate.

\begin{theo}\label{theo:kida}
Let $g,n\in\mathbb{N}$, and let $\Sigma$ be a connected, oriented surface of genus $g$ with $n$ points removed, with $3g+n-4>0$. Let $\Mod(\Sigma)\actson (X,\mu)$ be an ergodic measure-preserving essentially free action of $\Mod(\Sigma)$ on a standard probability space $(X,\mu)$, and let $H\actson (Y,\nu)$ be a weakly compact ergodic measure-preserving essentially free action of a countable group $H$ on a standard probability space $(Y,\nu)$.

If $L^\infty(X,\mu)\rtimes \Mod(\Sigma)$ is isomorphic to $L^\infty(Y,\nu)\rtimes H$, then the actions $\Mod(\Sigma)\actson (X,\mu)$ and $H\actson (Y,\nu)$ are stably conjugate.  
\qed
\end{theo}

We would like to conclude by making a few remarks about this statement. 

First, the notion of stable superrigidity from \cite{PV2} deals better with finite index phenomena. The action $G\actson (X,\mu)$ is \emph{stably $W^*_{wc}$-superrigid} if for every weakly compact measure-preserving ergodic essentially free action $H\actson (Y,\nu)$ on a standard probability space $(Y,\nu)$, if there exists an isomorphism $\theta$ from $L^\infty(X)\rtimes G$ to an augmentation $(L^\infty(Y)\rtimes H)^t$, then the actions $G\actson (X,\mu)$ and $H\actson (Y,\nu)$ are stably conjugate. In fact, we get that all ergodic measure-preserving essentially free actions of the mapping class group on standard probability spaces are $W^*_{wc}$-superrigid. In fact, the isomorphism $\theta$ can be described explicitly from a map realizing the stable conjugation.     

Second, we mention that a similar result also holds for related classes of groups. For example, one can let $G$ be a  direct product of mapping class groups: indeed, direct products of properly proximal groups are again properly proximal \cite[Proposition~4.10]{BIP}, and the orbit equivalence rigidity result from \cite{Kid} still holds for products of mapping class groups. Also, Chifan and Kida showed in \cite{CK} that many interesting subgroups of the mapping class groups which act nonelementarily on the curve graph -- such as the Torelli subgroup -- are rigid for measure equivalence, and therefore their ergodic actions are rigid for orbit equivalence. Such groups are properly proximal by Theorem~\ref{theo:mcg}, so an analogue of Theorem~\ref{theo:kida} also holds for these groups. 
 
\footnotesize 


\begin{flushleft}
Camille Horbez\\
Universit\'e Paris-Saclay, CNRS,  Laboratoire de math\'ematiques d'Orsay, 91405, Orsay, France \\
\emph{e-mail:}\texttt{camille.horbez@universite-paris-saclay.fr}\\[8mm]
\end{flushleft}

\begin{flushleft}
Jingyin Huang\\
Department of Mathematics\\
The Ohio State University, 100 Math Tower\\
231 W 18th Ave, Columbus, OH 43210, U.S.\\
\emph{e-mail:~}\texttt{huang.929@osu.edu}\\[8mm]
\end{flushleft}

\begin{flushleft}
Jean Lécureux\\
Universit\'e Paris-Saclay, CNRS,  Laboratoire de math\'ematiques d'Orsay, 91405, Orsay, France \\
\emph{e-mail:}\texttt{jean.lecureux@universite-paris-saclay.fr}\\
\end{flushleft}

\end{document}